\documentclass[12pt,a4paper]{article}

\RequirePackage[OT1]{fontenc}
\RequirePackage{amsthm,amsmath}
\RequirePackage[numbers]{natbib}
\RequirePackage[colorlinks,citecolor=blue,urlcolor=blue]{hyperref}
\usepackage{graphicx}

\usepackage[utf8]{inputenc}
\usepackage[english]{babel}
\usepackage{hyperref}
\usepackage{bookmark}
\usepackage{amsfonts}

\usepackage{amssymb}
\usepackage{enumitem}
\usepackage{color}
\usepackage{verbatim}
\usepackage{float}
\usepackage{bbm}
\usepackage[toc,page]{appendix}

\numberwithin{equation}{section}
\theoremstyle{plain}
\newtheorem{definition}{Definition}[section]

\newtheorem{prop}[definition]{Proposition}

\newtheorem{lemma}[definition]{Lemma}
\newtheorem{theo}[definition]{Theorem}
\newtheorem{corollar}[definition]{Corollary}

\title{Optimal Convergence Rates of Deep Neural Networks in a Classification Setting}
\author{Joseph T. Meyer\\{\url{joseph-theo.meyer@uni-heidelberg.de}}\\
Institute for Applied Mathematics, Heidelberg University,\\ Im Neuenheimer Feld 205, 69120 Heidelberg, Germany}

\begin{document}


\maketitle

\begin{abstract}
	 We establish optimal convergence rates up to a $\mathrm{log}$-factor for a class of deep neural networks in a classification setting under a restraint sometimes referred to as the Tsybakov noise condition. We construct classifiers in a general setting where the boundary of the bayes-rule can be approximated well by neural networks. Corresponding rates of convergence are proven with respect to the misclassification error. It is then shown that these rates are optimal in the minimax sense if the boundary satisfies a smoothness condition. Non-optimal convergence rates already exist for this setting. Our main contribution lies in improving existing rates and showing optimality, which was an open problem. Furthermore, we show almost optimal rates under some additional restraints which circumvent the curse of dimensionality. For our analysis we require a condition which gives new insight on the restraint used. In a sense it acts as a requirement for the "correct noise exponent" for a class of functions.   
\end{abstract}

\begin{tabbing}
networks\\
\emph{MSC2020 subject classifications:} 62C20, 62G05.
\end{tabbing}

\section{Introduction}
    
    We consider i.i.d. data $(Y_i,X_i)_{i=1}^n$ with $Y_i\in\{0,1\}$ and $X_i\in\mathbb{R}^d$. Our goal is to provide an estimator of the form $\hat{Y}=\mathbbm{1}(X\in \hat{G})$ where $\hat{G}$ is constructed with a neural network which approximates $Y$ well with respect to the misclassification error. We show optimal convergence rates under the following two conditions. First, the underlying distribution $\mathbb{Q}$ satisfies a noise condition as in \cite{tsybakov2004optimal} described below. Second, the boundary of the set
    \[
        G_{\mathbb{Q}}^*:=\Big\{x\ \Big|\ f_{\mathbb{Q}}(x)\geq\frac{1}{2}\Big\}
    \]
    with $f_\mathbb{Q}(x):=\mathbb{Q}(Y=1|X=x)$ satisfies certain regularity conditions.\newline
    
    \noindent Neural Networks have shown outstanding results in many classification tasks such as image recognition \cite{he2016deep}, language recognition \cite{collobert2008unified}, cancer recognition \cite{khan2001classification}, and other disease detection \cite{leibig2017leveraging}. Our work follows current approaches in the statistical literature to explain the success of neural networks, e.g. the impactful contributions  \cite{kohler2022estimation, schmidt2020nonparametric}. The objective is to fill a gap in the literature by proving optimal convergence rates in a specific setting which was also considered in \cite{kim2021fast}. We focus on deep feedforward neural networks with ReLU-activation functions. Deep networks have been considered in many theoretical articles \cite{kohler2022estimation, kohler2020statistical, petersen2018optimal, petersen2021optimal} and have proven useful in many applications \cite{schmidhuber2015deep, lecun2015deep}. Intuitively, we wish to approximate the set $G_\mathbb{Q}^*$ directly instead of approximating the regression function $f_\mathbb{Q}$. The classification setting we consider is similar to the setting given in \cite{mammen1999smooth, tsybakov2004optimal}. In particular, we assume that $\mathbb{Q}$ satisfies a noise condition which can be described as follows. For $\mathbb{Q}$-measurable sets $G_1,G_2$ define
    \begin{align*}
        d_{f_\mathbb{Q}}(G_1,G_2) & := \int_{G_1\Delta G_2}|2f_\mathbb{Q}(x)-1|\ \mathbb{Q}_X(\mathrm{d}x),\\
        d_\Delta(G_1,G_2) & :=\mathbb{Q}_X(G_1\Delta G_2).
    \end{align*}
    The condition then states that there exists a constant $\kappa\geq 1$ such that
    \begin{align}
        d_{f_\mathbb{Q}}(G,G^*_\mathbb{Q})\geq c_1d^\kappa_\Delta(G,G^*_\mathbb{Q}) \label{Characteristic Condition}
    \end{align}
    for some constant $c_1>0$ and all $G$. This requirement is sometimes referred to as the Tsybakov noice condition. It can be interpreted as a restraint on the probability distribution regarding regions close to the boundary where $f_\mathbb{Q}(x)=\frac{1}{2}$. Roughly speaking, it forces the mass to decay at a certain rate when one approaches this boundary. Using this, one can achieve rates approaching $n^{-1}$ for small $\kappa$, i.e. if there is not much mass in the region around $f_\mathbb{Q}(x)=\frac{1}{2}$. The condition has been used in many statistical articles considering classification such as \cite{audibert2007fast} and \cite{wu2005svm} who analyse support vector machines. Similarly to \cite{mammen1999smooth}, we show optimal convergence rates in the case where the boundary of $G_\mathbb{Q}^*$ satisfies certain regularity conditions, i.e. is similar to an element of a  Dudley class \cite{dudley1974metric}. More precisely, we consider sets which are slightly more general then the sets given in \cite{petersen2018optimal}. While many other approximation results using neural networks exist, see \cite{cybenko1989approximation} using sigmoid activation functions or \cite{yarotsky2017error} using piecewise linear functions among others, the methods used in \cite{petersen2018optimal} inspired us to obtain the results for our setting. The sets they consider have been used in many articles such as \cite{kim2021fast, petersen2021optimal}. As an estimator, we use a risk minimizer of the empirical version of the misclassification error. Precisely calculating this estimator involves finding a global minimum of a highly non-convex loss with respect to the parameters of a neural network. Typically, such calculations are not feasible in practice. Thus, the results we provide are theoretical in nature and do not have direct useful applications, as is typical for results of this kind \cite{kohler2020statistical, schmidt2020nonparametric}. From our point of view, the main value of current contributions is to show results such as consistency in situations which are typical for statisticians using relatively simple classes of neural networks. In time, the techniques developed may be used to show claims in cases which are closer to those encountered in reality and using classes of neural networks which are closer to those used in practise. 
    
    \noindent A lot of work has been done regarding consistency of feedforward deep neural networks. \cite{schmidt2020nonparametric} prove optimal convergence rates with respect to the uniform norm in a regression setting. Among others, similar results were given by \cite{imaizumi2019deep} for non-continuous regression functions with respect to the $L_2$-norm, \cite{kohler2021rate} who did not use a sparsity constraint, and \cite{barron1994approximation}. Regarding results for classification, \cite{petersen2021optimal} show convergence rates considering the misclassification error in a noiseless setting. Consistency results which include condition \eqref{Characteristic Condition} in the assumptions are given by \cite{bos2022convergence, kohler2020statistical, hu2021understanding}. In contrast to our approach, the previously mentioned articles attempt to estimate the regression function $f_\mathbb{Q}$ instead of directly estimating the set $G^*_\mathbb{Q}$. Additionally, while some obtain optimal convergence rates, the settings do not correspond to the setting given in \cite{tsybakov2004optimal}. In particular, the (optimal) convergence rates differ from ours in these papers. A very interesting contribution was made by \cite{kim2021fast} who consider an almost identical situation to ours, while their estimators differ. However, the rates they obtain are not optimal in the minimax sense.
    
\subsection{Contribution}

    Our contribution includes the following.
    \begin{itemize}
        \item First and foremost, Theorem \ref{Theorem Lower Bound} together with Corollary \ref{Corollar main} prove optimal convergence rates in the minimax sense for the setting described above. To the best of our knowledge, we are the first to obtain optimal convergence rates using neural networks corresponding to the setting given in \cite{tsybakov2004optimal} and thus close this gap in the literature.
        \item Theorem \ref{Theorem Main} establishes convergence rates in a general setting, where the boundary of the set $G_\mathbb{Q}^*$ can be well approximated by neural networks. This enables us to prove rates in a variety of settings. We use this theorem to prove optimal convergence rates under an additional constraint, which circumvents the curse of dimensionality, in the sense that the rates do not decrease exponentially in the dimension $d$.
        \item In order to prove the results stated here, we require a condition which together with condition \eqref{Characteristic Condition} forces $\kappa$ to be the "correct parameter" for the distribution $\mathbb{Q}$. We believe that this condition may bring new insights to condition \eqref{Characteristic Condition} and may be helpful in other situations where \eqref{Characteristic Condition} is required.
    \end{itemize}
\subsection{Outline}

    After introducing some notation, we rigorously introduce the problem at hand in Section \ref{General Convergence Results}. Here, we also provide some convergence results considering empirical risk minimizers with respect to arbitrary sets. These results are then used to prove our main consistency theorems regarding neural networks in Section \ref{Convergence Rates for Neural Networks}. Section  \ref{Lower Bound} includes the corresponding lower bounds followed by some concluding remarks in Section \ref{Concluding Remarks}.\newline 

\subsection{Notation}

    We introduce some general notation which is used throughout this article.\newline 
    
    \noindent For $x\in\mathbb{R}$, let $\lfloor x\rfloor:=\max\{k\in\mathbb{Z}\ |\ k\leq x\}$ and $\lceil x\rceil:=\min\{k\in\mathbb{Z}\ |\ k\geq x\}$. Let $\lambda$ be the Lebesgue measure. For a function $g:\Omega\subseteq\mathbb{R}^s\rightarrow\mathbb{R}$ and $k\in\mathbb{N}$ denote by
    \begin{align*}
        & \|g\|_\infty:=\sup_{x\in\Omega}|g(x)|,\ \ \ \|g\|_{L^k}:=\left(\int|g|^k\ \lambda(\mathrm{d}x)\right)^{\frac{1}{k}}
    \end{align*}
    the uniform-norm and the $L^k$-norm, respectively. Note that we omit the dependence on $\Omega$ in the notation. For $x\in\mathbb{R}^s$, let $\|x\|_2$ and $\|x\|_\infty$ be the euclidean-norm and the uniform-norm, respectively. For $j\in\{1,\dots,s\}$ let $$x_{-j}:=(x_1,\dots,x_{j-1},x_{j+1},\dots,x_s).$$ Additionally, let 
    \begin{align*}
        \mathcal{B}_r(x) & :=\{y\in\mathbb{R}^s\ |\ \|x-y\|_\infty\leq r\},\\
        \mathcal{B}^\circ_r(x) & :=\{y\in\mathbb{R}^s\ |\ \|x-y\|_\infty< r\}.
    \end{align*}
    For $a\in\mathbb{N}^s$ let $|a|:=\sum_{i=1}^sa_i$.\newline
    
    \noindent Now, let $\beta\in(0,\infty)$. Define $m:=\max\{k\in\mathbb{N}\ |\ k<\beta\}$ and $\omega:=\beta-m>0$. For $f\in\mathcal{C}\big([0,1]^s,\mathbb{R}\big)$ let
    \[
        \|f\|_{\mathcal{C}^\beta}:=\sum_{|\alpha|\leq m}\ \big\|\partial^\alpha f\|_\infty + \sum_{|\alpha|= m}\sup_{x\neq y} \frac{|\partial^\alpha f(x)-\partial^\alpha f(y)|}{|x-y|_\infty^\omega}
    \]
    be the Hoelder-norm. For $B>0$, define the class of Hoelder-continuous functions by
    \[
        \mathcal{F}_{\beta,B,s}:=\Big\{f\in\mathcal{C}\big([0,1]^s,\mathbb{R}\big) \ \Big|\ \|f\|_{\mathcal{C}^\beta}\leq B\Big\}.
    \]
    Let $G_1,G_2\subseteq\Omega$ be two subsets. We write 
    \[
        G_1\Delta G_2:=(G_1\backslash G_2)\cup (G_2\backslash G_1)
    \]
    for their symmetric difference and 
    \[
        \mathbbm{1}(x\in G_1):=\begin{cases}
        1,\ \text{for}\ x\in G_1,\\
        0,\ \text{otherwise}
        \end{cases}
    \]
    for the indicator function corresponding to $G_1$.

\section{General Convergence Results}\label{General Convergence Results}

    In this section, we state our results in a relatively general setting. The results on neural networks in the next section only consider the case where $\mathbb{Q}_X$ has a bounded density with respect to the Lebesgue measure.
    Our setup is similar to the binary classification setup of \cite{tsybakov2004optimal}. 
    
\subsection{Classification Setup}

    Let $(X_i,Y_i)_{i=1}^n$ be $i.i.d.$ observations distributed according to some probability measure $\mathbb{Q}$, where $X_i\in\mathbb{R}^d$ and $Y_i\in\{0,1\}$. Denote by $\mathbb{Q}_X$ the marginal probability distribution with respect to $X\in\mathbb{R}^d$. The goal is to predict $Y\in\{0,1\}$ when observing $X\in\mathbb{R}^d$, where $(X,Y)$ is distributed according to $\mathbb{Q}$ independently of $(X_i,Y_i)_{i=1}^n$ using classifiers of the form
    \[
        \hat{Y}:=\mathbbm{1}(X\in \hat{G})
    \]
    for some $\mathbb{Q}$-measurable set $\hat{G}\subseteq\mathbb{R}^d$. Note that a classifier is uniquely determined by $\hat{G}$. Performance is measured by the misclassification error
    \[
        R(\hat{G}):=\mathbb{P}\big(Y\neq\hat{Y}\big)=\mathbb{E}\Big[\big(Y-1(X\in \hat{G})\big)^2\Big].
    \]
    For $f_\mathbb{Q}(x):=\mathbb{E}[Y|X=x]=\mathbb{Q}(Y=1|X=x)$ the set
    \[
        G_{\mathbb{Q}}^*:=\Big\{x\ \Big|\ f_{\mathbb{Q}}(x)\geq\frac{1}{2}\Big\}
    \]
    is a so called bayes rule and thus minimizes the misclassification error. Classification can equivalently be seen as estimation of $G_\mathbb{Q}^*$ by the set $\hat{G}$, which is therefore equally referred to as classifier. For a $\mathbb{Q}$-measureable set $G\subseteq\mathbb{R}^d$ let
    \[
        R_n(G)=\frac{1}{n}\sum_{i=1}^n1\big(Y_i\neq 1(X_i\in G)\big)=\frac{1}{n}\sum_{i=1}^n\big(Y_i-1(X_i\in G)\big)^2
    \]
    be the empirical version of the misclassification error $R(G)$. We consider empirical risk minimization classifiers defined by
    \[
        \hat{G}_n:=\underset{G\in\mathcal{N}_n}{\mathrm{arg\ min}}\ R_n(G)
    \]
    where $\mathcal{N}_n$ is some finite collection of $\mathbb{Q}$-measurable sets for all $n\in\mathbb{N}$.
    
\subsection{Consistency Results}

    Proposition \ref{Theorem Mammen} establishes convergences rates for estimating $G^*_\mathbb{Q}$ using $\hat{G}_n$ under certain conditions on $\mathcal{N}_n$ and $\mathbb{Q}$. For the loss function, we consider a slight generalization of the misclassification error
    \[
        \mathbb{E}\big[\big(R(\hat{G}_n)-R(G_\mathbb{Q}^*)\big)^p\big]=\mathbb{E}\big[d^p_{f_\mathbb{Q}}(G_n,G_\mathbb{Q}^*)\big]
    \]
    for $p\geq 1$. The proposition is somewhat similar to Theorem 2 from  \cite{mammen1999smooth}. In contrast to our approach, they consider the discrimination of two probability distributions with underlying distribution functions and do not allow for non-optimal convergence rates. The proposition is an important component for the proof of our main theorem given in Section \ref{Convergence Rates for Neural Networks}. The proofs of this section can be found in Appendix \ref{Appendix - General Convergence Results}.

    \begin{prop}\label{Theorem Mammen}
        Let $\tau_n>0$ be a monotonically increasing sequence. Let $\mathfrak{Q}$ be a class of potential joint distributions $\mathbb{Q}$ of $(X,Y)$ and $\mathcal{N}_n$ be a collection of subsets of $\mathbb{R}^d$ for all $n\in\mathbb{N}$ such that the following conditions hold.
        \begin{enumerate}
            \item[(i)] For all $\mathbb{Q}\in\mathfrak{Q}$ all sets in $\bigcup_{n\in\mathbb{N}}\mathcal{N}_n$ and $G^*_\mathbb{Q}$ are $\mathbb{Q}$-measurable.
            \item[(ii)] There exists a constant $\kappa\geq 1$ such that
            \[
                d_{f_\mathbb{Q}}(G,G^*_\mathbb{Q})\geq c_1d^\kappa_\Delta(G,G^*_\mathbb{Q})
            \]
            for some constant $c_1>0$, all $G\in\bigcup_{n\in\mathbb{N}}\mathcal{N}_n$ and all $\mathbb{Q}\in\mathfrak{Q}$.
        \end{enumerate}
        Additionally, we assume that there is a constant $N_0\in\mathbb{N}$ such that for all $n\geq N_0$ the following holds.
        \begin{enumerate}
            \item[(iii)] There is a constant $c_2>0$ such that for all $\mathbb{Q}\in\mathfrak{Q}$ there is a $G\in\mathcal{N}_n$ with
            \[
                d_{f_\mathbb{Q}}(G,G^*_\mathbb{Q})\leq c_2\tau_n^{-\kappa}.
            \]
            \item[(iv)] There exist constants $c_3,\rho>0$ such that
            \[
                \log\big(|\mathcal{N}_n|\big)\leq c_3n^{\frac{\rho}{\rho+2\kappa-1}}.
            \]
        \end{enumerate}
        Then for all $p\geq 1$ we have
        \begin{align*}
            \underset{n\rightarrow\infty}{\lim\ \sup}\ & \underset{\mathbb{Q}\in\mathfrak{Q}}{\sup}\ \tilde{\tau}_n^{\kappa p}\ \mathbb{E}\big[d_{f_\mathbb{Q}}^p(\hat{G}_n,G^*_\mathbb{Q})\big]<\infty,\\
            \underset{n\rightarrow\infty}{\lim\ \sup}\ & \underset{\mathbb{Q}\in\mathfrak{Q}}{\sup}\ \tilde{\tau}_n^p\ \mathbb{E}\big[d_\Delta^p(\hat{G}_n,G^*_\mathbb{Q})\big]<\infty,
        \end{align*}
        where 
        \[ 
            \tilde{\tau}_n:=\min\{\tau_n,n^{\frac{1}{\rho+2\kappa-1}}\}
        \]
        for all $n\in\mathbb{N}$.
    \end{prop}
    
    Condition (i) is needed for all terms to be well defined. Condition (iii) states that the set in question must be well approximated by elements of $\mathcal{N}_n$. A sufficient assumption is that $\mathcal{N}_n$ is an $\epsilon$-net of $\big\{G_\mathbb{Q}^*\ \big|\ \mathbb{Q}\in\mathfrak{Q}\big\}$,
    where $\epsilon:=c_1\tau_n^{-\kappa}$. Together with (iv), this indirectly bounds the complexity of $\mathfrak{Q}$. If the class of sets $\{G^*_\mathbb{Q}\ |\ \mathbb{Q}\in\mathfrak{Q}\}$ is to large, one will not be able to find sets $\mathcal{N}_n$ that satisfy (iii) and (iv) at the same time. It is clear that the best rates are achieved with $\tau_n = n^{\frac{1}{\rho+2\kappa-1}}$. We do not use the same sequences in conditions (iii) and (iv) since one can prove non-optimal convergence rates using this version of the proposition.\newline 
    The second condition is the noise condition described in the introduction. Note that following \cite{tsybakov2004optimal}, for $\kappa>1$ condition (ii) holds if
    \begin{align*}
        \mathbb{P}\Big(\big|2f_\mathbb{Q}(X)-1\big|\leq t\Big)\leq ct^{\frac{1}{\kappa-1}}\label{equation tsyba condition}
    \end{align*}
    for all $t>0$ and some $c>0$. Roughly speaking, this forces the mass to decay at a certain rate when one approaches the boundary of $G_\mathbb{Q}^*$. Note that we use (ii) instead of this assumption since it is slightly more general and includes the case $\kappa=1$. Additionally, it appears more naturally in the proofs. Observing the alternative assumption, $\kappa=1$ corresponds to the case where there is no mass close to the boundary of $G_\mathbb{Q}^*$, meaning that $f_\mathbb{Q}$ does not take on values close to $\frac{1}{2}$. Using Proposition \ref{Theorem Mammen} one can achieve rates approaching $n^{-1}$ for small $\kappa,\rho$ i.e. if there is not much mass in the region around $f_\mathbb{Q}(x)=\frac{1}{2}$ and the complexity of $\mathcal{N}_n$, and consequently $\mathfrak{Q}$, is moderate. In contrast, the following proposition provides convergence rates if condition (ii) is not satisfied.
    
    \begin{prop}\label{Theorem Mammen 2}
        Let $\tau_n>0$ be a monotonically increasing sequence. Let $\mathfrak{Q}$ be a class of potential joint distributions $\mathbb{Q}$ of $(X,Y)$ and $\mathcal{N}_n$ be a collection of subsets of $\mathbb{R}^d$ for all $n\in\mathbb{N}$ such that the following conditions hold.
        \begin{enumerate}
            \item[(i)] For all $\mathbb{Q}\in\mathfrak{Q}$ all sets in $\bigcup_{n\in\mathbb{N}}\mathcal{N}_n$ and $G^*_\mathbb{Q}$ are $\mathbb{Q}$-measurable.
        \end{enumerate}
        Additionally, we assume that there is a constant $N_0\in\mathbb{N}$ such that for all $n\geq N_0$ the following holds.
        \begin{enumerate}
            \item[(ii)] There is a constant $c_2>0$ such that for all $n\in\mathbb{N}$ and $\mathbb{Q}\in\mathfrak{Q}$ there is a $G\in\mathcal{N}_n$ with
            \[
                d_{f_\mathbb{Q}}(G,G^*_\mathbb{Q})\leq c_2\tau_n^{-1}.
            \]
            \item[(iii)] There exist $c_3,\rho>0$ such that
            \[
                \log\big(|\mathcal{N}_n|\big)\leq c_3n^{\frac{\rho}{\rho+2}}.
            \]
        \end{enumerate}
        Then for all $p\geq 1$ we have
        \begin{align*}
            \underset{n\rightarrow\infty}{\lim\ \sup}\ & \underset{\mathbb{Q}\in\mathfrak{Q}}{\sup}\ \tilde{\tau}_n^{p}\ \mathbb{E}\big[d_{f_\mathbb{Q}}^p(\hat{G}_n,G^*_\mathbb{Q})\big]<\infty,
        \end{align*}
        where 
        \[
            \tilde{\tau}_n:=\min\{\tau_n,n^{\frac{\rho}{\rho+2}}\}
        \]
        for all $n\in\mathbb{N}$.
    \end{prop}
    
    Note that the requirement in conditon (ii) of Proposition \ref{Theorem Mammen 2} corresponds to requirement (iii) of \ref{Theorem Mammen} with $\kappa=1$. However, for $p=1$ the best rate achievable is of order $n^{-\frac{1}{p+2}}$, which is always slower $n^{-\frac{1}{2}}$. Proposition \ref{Theorem Mammen 2} provides rates in absence of condition (ii) of Proposition \ref{Theorem Mammen}. We do not claim optimality for these rates. 
    
\section{Convergence Rates for Neural Networks}\label{Convergence Rates for Neural Networks}

    We begin by shortly introducing neural networks. The idea is to use Proposition \ref{Theorem Mammen} to obtain optimal convergence rates up to a log factor. Neural networks are used to define a suitable class of sets $\mathcal{N}_n$ for every $n\in\mathcal{N}$.

\subsection{Definitions regarding Neural Networks}

    \begin{definition}
        Let $L,m_0,\dots,m_{L+1}\in\mathbb{N}$. For $i=1,\dots,L$, let $\sigma_i$ be a function
        \[
            \sigma_i:\mathbb{R}\rightarrow\mathbb{R}.
        \] 
        For $b=(b_1,\dots, b_{m_i})\in\mathbb{R}^{m_i}$, define a shifted $m_i$-dimensional version of $\sigma_i$ by 
        \[
            \sigma_{i,b}:\mathbb{R}^{m_i}\rightarrow\mathbb{R}^{m_i},\ \sigma_{i,b}(y_1,\dots,y_{m_i})=\big(\sigma_i(y_1-b_1),\dots,\sigma_i(y_{m_i}-b_{m_i})\big).
        \]
        A neural network with network architecture 
        \[
            \big(L,(m_0,\dots,m_{L+1}),(\sigma_1,\dots,\sigma_{L})\big)
        \]
        is a sequence 
        \[
            \Phi:=\big(W_1,b_1,\dots,W_{L},b_{L},W_{L+1}\big)
        \]
        where each $W_s\in\mathbb{R}^{m_s\times m_{s-1}}$ is a weight matrix and $b_s\in\mathbb{R}^{m_s}$ is a shift vector. The realization of a neural network $\Phi$ on a set $D\subseteq\mathbb{R}^{m_0}$ is the function
        \[
            R(\Phi):D\rightarrow\mathbb{R}^{m_{L+1}},\ R(\Phi)(x)=W_{L+1}\sigma_{L,b_L}W_L\cdot\cdot\cdot W_2\sigma_{1,b_1}W_1x.
        \]
        We denote by 
        \[
            \mathcal{N}_{L,m,\sigma}:=\big\{R(\Phi)\ |\  \Phi=\big(W_1,b_1,\dots,W_{L+1},b_{L+1}\ \big),\ W_s\in\mathbb{R}^{m_s\times m_{s-1}}, b_s\in\mathbb{R}^{m_s} \big\}
        \] 
        the set of realizations of neural networks with network architecture $(L,m,\sigma)$, where $m:=(m_0,\dots,m_{L+1})\in\mathbb{N}^{L+2}$ and $\sigma:=(\sigma_1,\dots,\sigma_{L})$.
    \end{definition}
    
    Typically, for $i,\dots,L$ the function $\sigma_i$ is called activation function and $\sigma_{i,b}$ is named shifted activation function. The constant $L$ denotes the number of hidden Layers. The values $d:=m_0$ and $m_{L+1}$ are the input and output dimensions, respectively. In this article, we are interested in the case where for $i=1,\dots,L$ the activation function in the $i$-th layer is the rectifier linear unit (ReLU)
    \[
        \sigma_i(x):=\max\{x,0\}.
    \]
    Additionally, if not further specified, we consider a compact domain $D:=[0,1]^d$ and a one dimensional output $m_{L+1}=1$. For the sake of completeness, we note that a network with L=0 layers is of the form $\Phi:=(W)$ for $W\in\mathbb{R}^{m_0\times m_1}$ and has realization $R(\Phi)(x)=Wx$. Note that, in general, the weights of a neural network $\Phi$, i.e. the entries of its shift vectors $(b_1,\dots,b_{L+1})$ and weight matrices  $(W_1,\dots,W_{L+1})$, are not uniquely determined by its realization $R(\Phi)$. In the following, for brevity, we occasionally introduce a network by defining its realization. In such a case, it is clear from the presentation of the realization which precise neural network is considered.\newline
    
    \noindent As a first step, we wish to introduce a suitable finite class of sets parameterized by neural networks and count the number of elements. We define these sets as $R(\Phi)^{-1}(1)$ where $\Phi$ is a realization of a neural network. Equivalently, we could have considered neural networks with a binary step function in the output layer or find a neural network $\tilde{\Phi}$ and define the approximating set $R(\tilde\Phi)^{-1}((0.5,1])$, which is closer to the idea that the realization of the neural network represents some sort of probability. Since this is not the idea of our approximation results, we stick to the version above. In order to obtain a finite class, we need to reduce the number of considered elements of $\mathcal{N}_{L,m,\sigma}$ while maintaining reasonable approximating capabilities. A typical approach in the theoretical literature is to use a sparsity constraint. For $s>1$ we therefore only consider realizations of neural nets which have at most $s$ nonzero weights. If $s$ is the total number of nonzero weights, we say that the network has sparsity $s$. Additionally, we assume all weights to be elements of the set 
    \[
        \mathcal{W}_c:=\big\{k2^{-c}\ \big|\ c\in\mathbb{N},\ k\in\{-2^c, -2^c+1,\dots,2^c-1,2^c\}\big\}.
    \]
    Thus, we only consider weights $|w|\leq 1$. Concluding, we use the following notation to describe the collection of sets we are interested in.
    
    \begin{definition}\label{N_lsc}
        Let $L_0,c\in\mathbb{N}$ and $s_0>1$ be fixed. Denote by $\tilde{\mathcal{N}}_{L_0,s_0,c}$ the set of realizations of neural networks with $d$ dimensional input, one dimensional output, at most $L_0$ layers, ReLU activation functions and sparsity at most $s_0$, where all weights are elements of $\mathcal{W}_c$. The class of corresponding sets given by neural networks is then
        \[
            \mathcal{N}_{L_0,s_0,c}:=\big\{R(\Phi)^{-1}(1)\subseteq [0,1]^d,\ \big|\ \Phi\in \tilde{\mathcal{N}}_{L_0,s_0,c}\big\}.
        \]
    \end{definition}
    
    Note that the requirements from Definition \ref{N_lsc} allow for realizations of neural networks with arbitrary hidden layer dimensions $(m_1,\dots,m_L)\in\mathbb{N}^L$. However, it is easy to see that every element of $\tilde{\mathcal{N}}_{L_0,s_0,c}$ is a realization of a neural net which satisfies the properties described in the Definition and $m_i\leq s_0$ for all $i\in\{1,\dots, L_0\}$. Using this, we receive an upper bound on the number of elements of $\mathcal{N}_{L_0,s_0,c}$ by counting the number of corresponding neural networks. Thus, the following bound is independent of the choice of activation functions $\sigma$.
    
    \begin{lemma} \label{Lemma Anzahl der Elemente}
        For $s_0>1$ and $L_0,c\in\mathbb{N}$ let $\mathcal{N}_{L_0,s_0,c}$ be the class of sets introduced in Definition \ref{N_lsc}. We have an upper bound on the number of elements given by
        \[
            \big|\mathcal{N}_{L_0,s_0,c}\big|\leq \big((ds_0+\min\{s_0,L_0\}(s_0+1)^2)2^{c+2}\big)^{s_0}.
        \]
    \end{lemma}
    
    \begin{proof}
        First of all, if $s_0\leq L_0$, clearly only the last $s_0$ layers have influence on the realization of a neural network. Thus, an upper bound is given by counting the number of neural nets with at most $\min\{s_0,L_0\}$ layers, at most sparsity $s_0$, weights in $\mathcal{W}_c$ and $m_i\leq s_0$ for all $i\in\{1,\dots, L\}$. Each weight can take on $|\mathcal{W}_c|=2^{c+1}+1$ different values. The total number of weights can be bounded by
        \begin{align*}
            m_0m_1+\sum_{i=1}^{\min\{s_0,L\}}(m_i+1)m_{i+1} & \leq ds_0+\sum_{i=1}^{\min\{s_0,L_0\}}(s_0+1)s_0\\
            & \leq ds_0+\min\{s_0,L_0\}(s_0+1)^2\\
            & =:V.
        \end{align*}
        Note that if $s_0\leq L_0$, the input dimension does not influence the outcome. Therefore, there are at most
        \[
            \binom{V}{s_0}\leq V^{s_0}
        \]
        possible combinations to pick $s_0$ (possibly) nonzero weights. Thus
        \[
            \big|\mathcal{N}_{L_0,s_0,c}\big|\leq V^{s_0}\big(2^{c+1}+1\big)^{s_0}\leq \big(V2^{c+2}\big)^{s_0}.
        \]
    \end{proof}
    
\subsection{Conditions on the Bayes-Rule}
    
    In order to define the set of probability distributions we consider for approximation, we restrict the possible bayes rules. We then add a smoothness condition to the function $f_\mathbb{Q}$ near the boundary of the respective bayes rule. Intuitively, the boundary should satisfy some kind of smoothness condition so that it can be approximated by neural networks. Additionally, the set must be discretizable in some sense. When using $\mathcal{F}=\mathcal{F}_{\beta,B,d-1}$ the class of sets we use is similar to a class defined in \cite{petersen2018optimal}. Note, that the class used here is larger. This version depends on a set $\mathcal{F}$ which represents a class of boundary functions. The idea is that we can obtain different convergence rates for different classes using the same procedure. 

    \begin{definition}\label{Definition Sets}
        Let $\beta\geq 0$, $B>0$ and $d\in\mathbb{N}$ with $d\geq 2$. Additionally, let $r\in\mathbb{N}$, $\epsilon_1,\epsilon_2>0$, $\mathbb{Q}$ be a probability measure on $[0,1]^d\times\{0,1\}$ and $\mathcal{F}$ be a class of functions
        \[
            \gamma:[0,1]^{d-1}\rightarrow\mathbb{R}.
        \]  
        Define
        \begin{align*}
            \mathcal{I} & :=\{1,\dots,d\}\times\{-1,1\},\\
            \mathcal{D} & :=\bigg\{\prod_{i=1}^d[a_i,b_i]\ \bigg|\ 0\leq a_i<b_i\leq 1\bigg\}.
        \end{align*}
        and
        \begin{align*}
            \mathcal{H}_{\beta,B} & :=\big\{g\in\mathcal{F}_{\beta,B,1}\ \big|\  \forall \alpha< \beta :  \partial^\alpha g(0)=0 \big\}
        \end{align*}
        if $\beta>0$. Let $\mathcal{K}_{\mathbb{Q},\beta,B,\epsilon_1,\epsilon_2,r,d}^\mathcal{F}$ be the class of all subsets $H=H_1\cup\dots\cup H_u\subseteq [0,1]^d$ for $u\in\{0,\dots,r\}$ such that for $\nu=1,\dots,u$ there exist $(j_\nu,\iota_\nu)\in\mathcal{I},\  \gamma_\nu\in\mathcal{F},\ D_\nu=\prod_{i=1}^d[a_i^\nu,b_i^\nu]\in\mathcal{D}$ with the following properties.
        \begin{enumerate}
            \item For all $\nu=1,\dots,u$ we have
            \[
                H_\nu=D_\nu\cap \{x\in[0,1]^d\ |\ \iota_\nu x_{j_\nu}\leq \gamma_\nu(x_{-j_\nu})\}.
            \]
            \item $\lambda(D_{\nu_1}\cap D_{\nu_2})=0$ for $\nu_1\neq \nu_2$.
            \item If $\beta>0$, the following holds. For $\nu=1,\dots,u$ and all $x\in D_\nu\cap \partial H$ there exists $g_{\nu,x}\in\mathcal{H}_{\beta,B}$ such that for $y_{j_\nu}\in[\max\{0,x_{j_\nu}-\epsilon_1\},\min\{1,x_{j_\nu}+\epsilon_1\}]$ we have
            \begin{align*}
                |2f_\mathbb{Q}(y)-1| & \leq g_{\nu,x}(x_{j_\nu}-y_{j_\nu})\ \ \text{for}\ x_{j_\nu}-\epsilon_1\leq y_{j_\nu} \leq x_{j_\nu},\\
                |2f_\mathbb{Q}(y)-1| & \leq g_{\nu,x}(y_{j_\nu}-x_{j_\nu})\ \ \text{for}\ x_{j_\nu}\leq y_{j_\nu} \leq x_{j_\nu}+\epsilon_1,
            \end{align*}
            where $y=(x_1,\dots,x_{j_\nu-1},y_{j_\nu},x_{j_\nu+1},\dots,x_d)$.
            \item $b_{j_\nu}^\nu-a_{j_\nu}^\nu\geq\epsilon_2$ for all $\nu=1,\dots,u$.
        \end{enumerate}
    \end{definition} 
    
    \begin{figure}
        \begin{center}
            \includegraphics[scale=0.3]{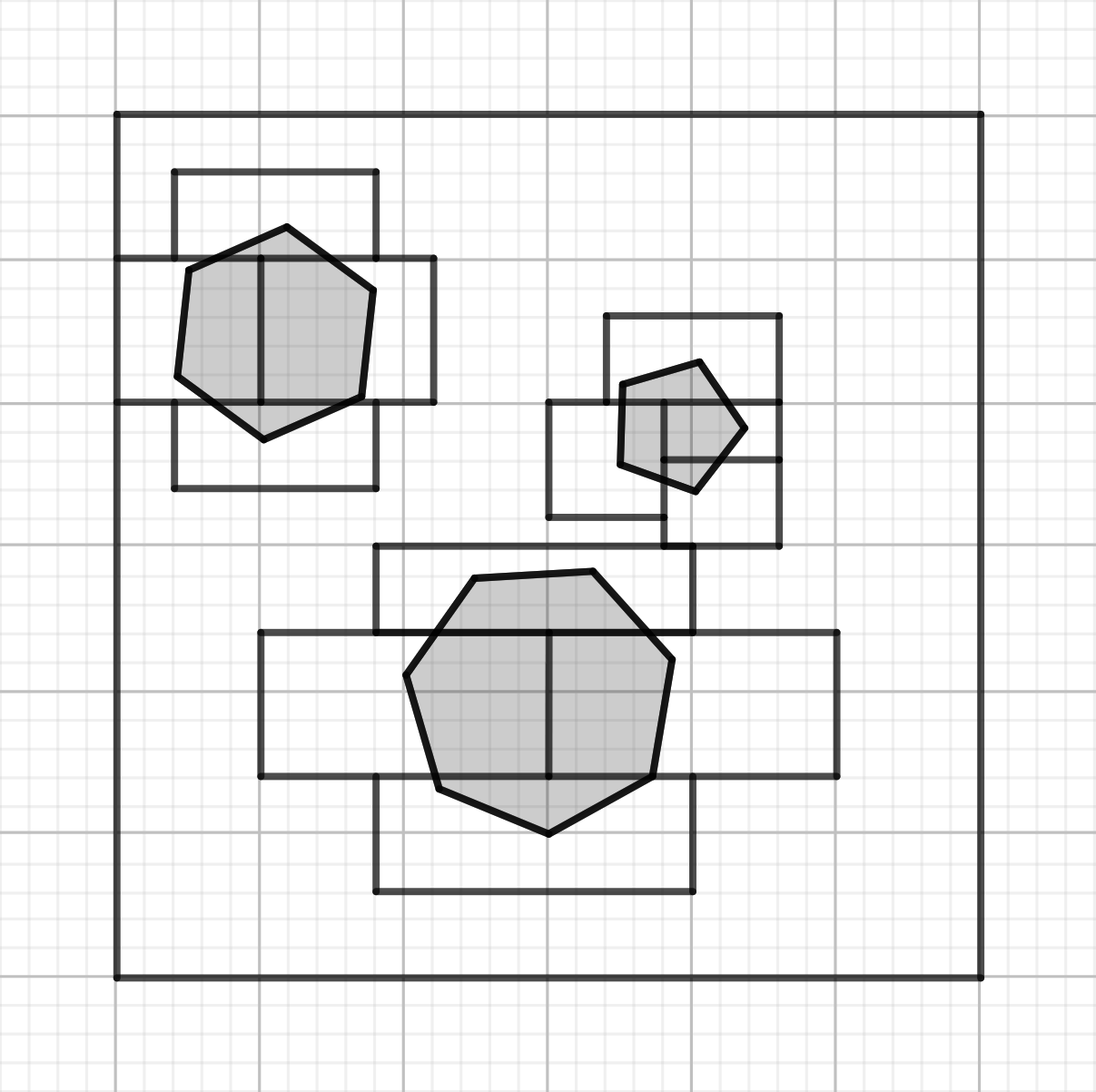}
        \end{center}
		\caption{Example for an element of $\mathcal{K}_{\mathbb{Q},\beta,B,\epsilon_1,\epsilon_2,12,2}^\mathcal{F}$, where $\mathcal{F}$ is the set of piecewise linear functions. The grey objects represent the set. The small boxes represent a possible choice for $D_1,\dots,D_{12}$.} \label{fig:Menge}
	\end{figure}
    
    Note that $g\in\mathcal{H}_{\beta,B}$ implies $g(0)=\partial ^0g(0)=0$. The idea is to use sets defined by realizations of neural networks to approximate $G^*_\mathbb{Q}\in \mathcal{K}_{\mathbb{Q},\beta,B,\epsilon_1,\epsilon_2,r,d}^\mathcal{F}$ from Definition $\ref{Definition Sets}$ for a suitable class $\mathcal{F}$ in order to apply Proposition \ref{Theorem Mammen}. Figure \ref{fig:Menge} shows an example for an element of $\mathcal{K}_{\mathbb{Q},\beta,B,\epsilon_1,\epsilon_2,12,2}^\mathcal{F}$, where $\mathcal{F}$ is the set of piecewise linear functions. The definition contains an additional condition on the function $f_\mathbb{Q}$ close to the boundary of $G^*_\mathbb{Q}$. Following the intuition mentioned in \cite{mammen1999smooth}, for $\beta>0$ condition (ii) in Theorem 2.1 means that  $f_\mathbb{Q}$ acts like $x^{\beta}$ close to the boundary of $G^*_\mathbb{Q}$, where $\kappa=1+\beta$. More precisely, condition (ii) requires that $f_\mathbb{Q}$ does not increase slower than $x^{\beta}$. In order to prove the combination of conditions (iii) and (iv), we require that $\beta$ is the correct rate, meaning that $f_\mathbb{Q}$ does not increase faster than $x^{\beta}$. In Section \ref{Lower Bound} we prove that this condition does not lower the complexity of the problem. Thus, the rates obtained by \cite{mammen1999smooth} are still optimal.
    
\subsection{Main Theorems}
    
    We begin by stating the central result of this article. We then use this result to show consistency results for more specific cases. The rates we obtain in Theorem \ref{Theorem Main} are optimal up to a $\mathrm{log}$ factor. In the following, all proofs of this section are given in Appendix \ref{Appendix - Convergence Rates for Neural Networks}.
    
    \begin{theo}\label{Theorem Main}
        Let $\beta\geq 0$, $B,\rho>0$ and $d\in\mathbb{N}$ with $d\geq 2$. Let $\mathcal{F}$ be a set of functions 
        \[
            \gamma:[0,1]^{d-1}\rightarrow\mathbb{R}
        \]
        such that the following holds. There exist $\epsilon_0,C_1,C_2>0$ and $C_3,C_4\in\mathbb{N}$ such that for any $\gamma\in \mathcal{F}$ and any $\epsilon\in(0,\epsilon_0)$ there is a neural network $\Phi$ with $L\leq L_0(\epsilon):=C_1 \lceil\log(\epsilon^{-1})\rceil$ layers, sparsity $s\leq s_0(\epsilon):=C_2 \epsilon^{-\rho}\log(\epsilon^{-1})$ and weights in $\mathcal{W}_c$ with $c=c_0(\epsilon):=C_3+C_4\lceil\log(\epsilon^{-1})\rceil$ such that
        \[
            \|R(\Phi)(x)-\gamma\|_\infty\leq\epsilon.
        \]
        Define $\kappa:=1+\beta$ and let $\mathfrak{Q}$ be a class of potential joint distributions $\mathbb{Q}$ of $(X,Y)$ such that the following conditions hold. 
        \begin{itemize}
            \item[(a)] There is a constant $M>1$ such that for all $\mathbb{Q}\in\mathfrak{Q}$ the marginal distribution of $\mathbb{Q}_X$ has a Lebesgue density bounded by $M$.
            \item[(b)] There are constants $r\in\mathbb{N}$ and $\epsilon_1,\epsilon_2>0$ such that for all $\mathbb{Q}\in\mathfrak{Q}$ the bayes rule satisfies $G_{\mathbb{Q}}^*\in \mathcal{K}_{\mathbb{Q},\beta,B,\epsilon_1,\epsilon_2,r,d}^\mathcal{F}$.
            \item[(c)] We have
            \[
                d_{f_\mathbb{Q}}(G,G^*_\mathbb{Q})\geq c_1d^\kappa_\Delta(G,G^*_\mathbb{Q})
            \]
            for some constant $c_1>0$, all $G\in\mathcal{N}$ and all $\mathbb{Q}\in\mathfrak{Q}$, where
            \[
                \mathcal{N}:=\bigcup_{n\in\mathbb{N}}\mathcal{N}_{n,n,n}
            \]
            is the class of sets corresponding to any neural network. 
        \end{itemize}
        Let
        \[
            \tau_n:=\frac{n^{\frac{1}{2\kappa+\rho-1}}}{\log^{\frac{2}{\rho}}(n)}.
        \]
        Then there exist constants $C_1',C_2'>0$ and $C_3'\in\mathbb{N}$  such that for all $p\geq 1$ we have
        \begin{align*}
            \underset{n\rightarrow\infty}{\lim\ \sup}\ & \underset{\mathbb{Q}\in\mathfrak{Q}}{\sup}\ \tau_n^{p\kappa}\ \mathbb{E}\big[d_{f_\mathbb{Q}}^p(\hat{G}_n,G^*_\mathbb{Q})\big]<\infty,\\
            \underset{n\rightarrow\infty}{\lim\ \sup}\ & \underset{\mathbb{Q}\in\mathfrak{Q}}{\sup}\ \tau_n^p\ \mathbb{E}\big[d_\Delta^p(\hat{G}_n,G^*_\mathbb{Q})\big]<\infty, 
        \end{align*}
        where
        \[
            \hat{G}_n:=\underset{G\in\mathcal{N}_n}{\mathrm{arg\ min}}\ R_n(G).
        \]
        with $\mathcal{N}_n=\mathcal{N}_{C_1'L_0(\tau_n^{-1}),C_2's_0(\tau_n^{-1}),C_3'c_0(\tau_n^{-1})}$.
    \end{theo}

\subsection{Results for Regular Boundaries}

    We can now prove results for specific classes of sets $\mathcal{F}$ to obtain convergence results. A first important example is the class $\mathcal{F}_{\beta,B,d}$. The following Lemma is a consequence of Theorem 5 in  \cite{schmidt2020nonparametric}.
    
    \begin{lemma}\label{Lemma Theorem 5}
        Let $\beta,B>0$ and $d\in\mathbb{N}$. Then there exist $\epsilon_0,c_1,c_2>0, c_3,c_4\in\mathbb{N}$ such that the following holds. For any function $\gamma\in\mathcal{F}_{\beta,B,d}$ and any $\epsilon\in(0,\epsilon_0)$, there exists a neural network $\Phi$ with $L\leq L_0(\epsilon):=c_1 \lceil\log(\epsilon^{-1})\rceil$ layers, sparsity $s\leq s_0(\epsilon):=c_2 \epsilon^{-\frac{d}{\beta}}\log(\epsilon^{-1})$ and weights in $\mathcal{W}_c$ with $c=c_0(\epsilon):=c_3+c_4\lceil\log(\epsilon^{-1})\rceil$ such that
        \[
            \|R(\Phi)(x)-\gamma\|_\infty\leq\epsilon.
        \]
    \end{lemma}
    
    \begin{corollar}\label{Corollar main}
        Let $\beta_1\geq 0$, $B_1,\beta_2,B_2>0$ and $d\in\mathbb{N}$ with $d\geq 2$. Let $\mathfrak{Q}$ be a class of potential joint distributions $\mathbb{Q}$ of $(X,Y)$. Assume (a),(b),(c) from Theorem \ref{Theorem Main} hold with $\rho:=\frac{d-1}{\beta_2}$, $\beta:=\beta_1$, $B:=B_1$ as well as $\mathcal{F}:=\mathcal{F}_{\beta_2,B_2,d-1}$. Let
        \[
            \tau_n:=\frac{n^{\frac{1}{2\kappa+\rho-1}}}{\log^{\frac{2}{\rho}}(n)}.
        \]
        Then there exist constants $C_1',C_2'>0$ and $C_3'\in\mathbb{N}$  such that for all $p\geq 1$ we have
        \begin{align*}
            \underset{n\rightarrow\infty}{\lim\ \sup}\ & \underset{\mathbb{Q}\in\mathfrak{Q}}{\sup}\ \tau_n^{p\kappa}\ \mathbb{E}\big[d_{f_\mathbb{Q}}^p(\hat{G}_n,G^*_\mathbb{Q})\big]<\infty,\\
            \underset{n\rightarrow\infty}{\lim\ \sup}\ & \underset{\mathbb{Q}\in\mathfrak{Q}}{\sup}\ \tau_n^p\ \mathbb{E}\big[d_\Delta^p(\hat{G}_n,G^*_\mathbb{Q})\big]<\infty, 
        \end{align*}
        where
        \[
            \hat{G}_n:=\underset{G\in\mathcal{N}_n}{\mathrm{arg\ min}}\ R_n(G)
        \]
        with $\mathcal{N}_n=\mathcal{N}_{C_1'L_0(\tau_n^{-1}),C_2's_0(\tau_n^{-1}),C_3'c_0(\tau_n^{-1})}$ and $L_0,s_0,c_0$ from Lemma \ref{Lemma Theorem 5}.
    \end{corollar}
    
    Corollary \ref{Corollar main} together with Theorem \ref{Theorem Lower Bound} from the next chapter prove optimal convergence rates, which was the main goal of this paper.\newline
    
    \noindent Following up, note that the rates we receive from Corollary \ref{Corollar main} are affected by the curse of dimensionality. Observe that the rates obtained by Theorem \ref{Theorem Main} are influenced by condition (c) on the one hand and the ability of neural networks to approximate sets in $\mathcal{F}$ on the other. The dependence on the dimension $d$ in Corollary \ref{Corollar main} comes from the latter. Thus, a natural approach to circumvent the curse of dimensionality is to approximate a smaller set $\mathcal{F}$. Intuitively, it is clear that without strong restrictions on the distribution we can only overcome the curse if the complexity of the boarders of the sets we approximate is small enough so that they themselves can overcome the curse. In the literature, many different sets are considered which infer useful approximation capabilities of neural networks. Here, we use a class of sets introduced by  \cite{schmidt2020nonparametric} which is close to class $\mathcal{F}_{\beta,B,d}$. 
    \begin{definition}\label{Definition curse}
        Let $r\in\mathbb{N}$, $t\in\mathbb{N}^{r}$, $d\in\mathbb{N}^{r+1}$, $\beta\in\mathbb{R}^r$ and $B>0$ with $t_i\leq d_i$,  $\beta_i>0$ for $i=1,\dots,r$, $d_{r+1}=1$. Define 
        \begin{align*}
            \mathcal{G}_{r,t,\beta,B,d}:=\big\{ \gamma = \gamma_r\circ\dots\circ\gamma_1\ \big|\ & \gamma_i=(\gamma_{ij}\circ\iota_{ij})_{j=1}^{d_{i+1}},\ \gamma_{ij}\in\mathcal{F}_{\beta_i,B,t_i},\ \iota_{ij}\in\mathcal{ID}_{i},\\
            & \gamma_{ij}:[0,1]^{t_{i}}\rightarrow[0,1]\ \text{for}\ i=1,\dots,r-1,\\
            & \gamma_{r1}:[0,1]^{t_{r}}\rightarrow\mathbb{R} \big\},
        \end{align*}
        where 
        \begin{align*}
            \mathcal{ID}_{i}=\big\{\iota:[0,1]^{d_{i}}\rightarrow[0,1]^{t_i}\ \big|\ \iota(x)=(x_{i_1},\dots,x_{i_{t_i}}),\ i_j\in\{1,\dots,d_i\} \big\}.
        \end{align*}
    \end{definition}
    Instead of requiring that $\gamma_{ij}$ is supported on $[0,1]^{t_i}$, we could have used the condition that $\gamma_{ij}$ is supported on $\prod_{k=1}^{t_i}[a_k,b_k]$ for some values $a_k,b_k\in\mathbb{R}$. However, this does not enlarge the class considerably. It can easily be seen that we can instead increase the bound $B$ to find an even larger class. The idea for using the set $\mathcal{G}_{r,t,\beta,B,d}$ is that its complexity does not depend on the input dimension $d_1$, but only on the most difficult component to approximate. The complexity of the components depend on their effective dimension $t_i$ and their implied smoothness. As described by  \cite{schmidt2020nonparametric}, the correct smoothness parameter to consider is
    \[
        \beta_i^*:=\beta_i\prod_{k=i+1}^r\min\{\beta_k,1\}.
    \]
    Examples for sets that can profit from Definition \ref{Definition curse} are additive models ($r=1$, $t_1=1$), interaction models of order $k$ ($r=1$, $t_1=k$), or multiplicative models (they are a subset of $\mathcal{G}_{r,t,\beta,B,d}$ when $r=\log_2(d)+1$, $t_i=2$ for all $i$). Next, our goal is to establish a convergence result when the set of boundary functions is $\mathcal{G}_{r,t,\beta,B,d}$. Similarly to the approach above, we first provide a lemma which provides approximation results using neural networks.
    
    \begin{lemma}\label{Lemma curse}
        Let $r\in\mathbb{N}$, $t\in\mathbb{N}^{r}$, $d\in\mathbb{N}^{r+1}$, $\beta\in\mathbb{R}^r$ and $B>0$ with $t_i\leq d_i$,  $\beta_i>0$ for $i=1,\dots,r$, $d_{r+1}=1$. Let $\mathcal{G}_{r,t,\beta,B,d}$ be defined as in Definition \ref{Definition curse} and define
        \[
            \rho:=\underset{i=1,\dots,r}{\max}\left(\frac{t_i}{\beta_i^*}\right).
        \] 
        Then there exist $\epsilon_0,c_1,c_2>0, c_3,c_4\in\mathbb{N}$ such that the following holds. For any function $\gamma\in\mathcal{G}_{r,t,\beta,B,d}$ and any $\epsilon\in(0,\epsilon_0)$, there exists a neural network $\Phi$ with $L\leq L_0(\epsilon) := c_1 \lceil\log(\epsilon^{-1})\rceil$ layers, sparsity $s\leq s_0(\epsilon):=c_2 \epsilon^{-\rho}\log(\epsilon^{-1})$ and weights in $\mathcal{W}_c$ with $c=c_0(\epsilon):=c_3+c_4\lceil\log(\epsilon^{-1})\rceil$ such that
        \[
            \|R(\Phi)(x)-\gamma\|_\infty\leq\epsilon.
        \]
    \end{lemma}
    
    The following corollary establishes the corresponding convergence result. Theorem \ref{Theorem Lower Bound curse} provides the lower bound in the case where $t_i\leq \min\{d_1,\dots,d_i\}$. 
    
    \begin{corollar}\label{Corollar curse}
        Let $r_2\in\mathbb{N}$, $t\in\mathbb{N}^{r_2}$, $d\in\mathbb{N}^{r_2+1}$, $\beta_1\geq 0$, $\beta_2\in\mathbb{R}^{r_2}$, and $B_1,B_2>0$ with  $\beta_{2,i}>0$ for $i=1,\dots,r_2$, $d_{r_2+1}=1$. Additionally, $t_1<d_1$ and $t_i\leq d_i$ for $i\neq 1$. Define $d'\in\mathbb{N}^{r_2+1}$ with $d'_1=d_1-1$, $d'_i=d_i$ for $i\neq 1$ and 
        \[
            \rho:=\max_{i=1,\dots,r_2}\ \frac{t_i}{\beta_{2,i}^*}
        \] 
        Define $\kappa:=1+\beta_1$ and let $\mathfrak{Q}$ be a class of potential joint distributions $\mathbb{Q}$ of $(X,Y)$ such that the following conditions hold. 
        \begin{itemize}
            \item[(a)] There is a constant $M>1$ such that for all $\mathbb{Q}\in\mathfrak{Q}$ the marginal distribution of $\mathbb{Q}_X$ has a Lebesgue density bounded by $M$.
            \item[(b)] There are constants $r_1\in\mathbb{N}$ and $\epsilon_1,\epsilon_2>0$ such that for all $\mathbb{Q}\in\mathfrak{Q}$ the bayes rule satisfies $G_{\mathbb{Q}}^*\in \mathcal{K}_{\mathbb{Q},\beta_1,B_1,\epsilon_1,\epsilon_2,r_1,d_1}^\mathcal{F}$ with
            \[
                \mathcal{F}:=\mathcal{G}_{r_2,t,\beta_2,B_2,d'}.
            \]
            \item[(c)] We have
            \[
                d_{f_\mathbb{Q}}(G,G^*_\mathbb{Q})\geq c_1d^\kappa_\Delta(G,G^*_\mathbb{Q})
            \]
            for some constant $c_1>0$, all $G\in\mathcal{N}$ and all $\mathbb{Q}\in\mathfrak{Q}$, where
            \[
                \mathcal{N}:=\bigcup_{n\in\mathbb{N}}\mathcal{N}_{n,n,n}
            \]
            is the class of sets corresponding to any neural network.
        \end{itemize}
        Let
        \[
            \tau_n:=\frac{n^{\frac{1}{2\kappa+\rho-1}}}{\log^{\frac{2}{\rho}}(n)}.
        \]
        Then there exist constants $C_1',C_2'>0$ and $C_3'\in\mathbb{N}$  such that for all $p\geq 1$ we have
        \begin{align*}
            \underset{n\rightarrow\infty}{\lim\ \sup}\ & \underset{\mathbb{Q}\in\mathfrak{Q}}{\sup}\ \tau_n^{p\kappa}\ \mathbb{E}\big[d_{f_\mathbb{Q}}^p(\hat{G}_n,G^*_\mathbb{Q})\big]<\infty,\\
            \underset{n\rightarrow\infty}{\lim\ \sup}\ & \underset{\mathbb{Q}\in\mathfrak{Q}}{\sup}\ \tau_n^p\ \mathbb{E}\big[d_\Delta^p(\hat{G}_n,G^*_\mathbb{Q})\big]<\infty, 
        \end{align*}
        where
        \[
            \hat{G}_n:=\underset{G\in\mathcal{N}_n}{\mathrm{arg\ min}}\ R_n(G).
        \]
        with $\mathcal{N}_n=\mathcal{N}_{C_1'L_0(\tau_n^{-1}),C_2's_0(\tau_n^{-1}),C_3'c_0(\tau_n^{-1})}$.
    \end{corollar}
    
    Note that Corollary \ref{Corollar curse} is a generalisation of Corollary \ref{Corollar main}. The rate now depends on $\rho$ which in turn depends on $t_1,\dots,t_{r_2}$ instead of the input dimension $d_1$. Clearly, the effective dimensions $t_i$ can be much smaller then the input dimension $d_1$, for example, when the boundary functions come from an additive function.
    
\section{Lower Bound}\label{Lower Bound}
    
    We now establish lower bounds on the convergence rates from corollaries \ref{Corollar main} and $\ref{Corollar curse}$. Note that the lower bounds also prove that the rates obtained in Theorem \ref{Theorem Main} can not be improved up to a $\mathrm{log}$-factor. Since Corollary $\ref{Corollar curse}$ is a generalisation of $\ref{Corollar main}$, we only have to prove a lower bound for the setting given in the former. For clarity, we provide both statements. The proofs of this section can be found in Appendix \ref{Appendix - Lower Bound}.
    
    \begin{theo}\label{Theorem Lower Bound}
         Let $\beta_1\geq 0$, $B_1,\beta_2,B_2,\rho>0$ and $d\in\mathbb{N}$ with $d\geq 2$.  Let $\mathfrak{Q}$ be the class of all potential joint distributions $\mathbb{Q}$ of $(X,Y)$ such that (a),(b) from Theorem \ref{Theorem Main} hold with $\rho:=\frac{d-1}{\beta_2}$, $\kappa:=1+\beta_1$, $\beta:=\beta_1$, $B:=B_1$, $\mathcal{F}:=\mathcal{F}_{\beta_2,B_2,d-1}$ and some $M> 1$, $r\in\mathbb{N}$, $\epsilon_1,\epsilon_2>0$. Let (c) hold with $c_1>0$ large enough and set
        \[
            \tau_n:=n^{\frac{1}{2\kappa+\rho-1}}.
        \]
        Then
        \begin{align*}
            \underset{n\rightarrow\infty}{\lim\ \inf}\ \underset{G_{n}\in\mathfrak{G}}{\inf}\ \underset{\mathbb{Q}\in\mathfrak{Q}}{\sup}\ \tau_n^p\mathbb{E}\left[d^p_\Delta(G_{n},G^*_\mathbb{Q})\right] & >0,\\
            \underset{n\rightarrow\infty}{\lim\ \inf}\ \underset{G_{n}\in\mathfrak{G}}{\inf}\ \underset{\mathbb{Q}\in\mathfrak{Q}}{\sup}\ \tau_n^{p\kappa}\mathbb{E}\left[d^p_{f_\mathbb{Q}}(G_{n},G^*_\mathbb{Q})\right] & >0
        \end{align*}
        for every $p\geq 0$, where $\mathfrak{G}$ contains all estimators depending on the data $(X_1,Y_1),\dots,(X_n,Y_n)$.
    \end{theo}

    Intuitively, $B_1$ bounds the factor of the $x^{\beta_2}$-term of $f_\mathbb{Q}$ close to the boundary from above. On the other hand, $c_1$ bounds this term from below. Thus, not every combination of $B_1,c_1>0$ is possible. We prove Theorem \ref{Theorem Lower Bound} for large $c_1>0$. We do not provide the exact ratio of $B$ and $c_1$ required since it is not important for the statement. Lastly, the lower bound corresponding to Corollary \ref{Corollar curse} is given.

    \begin{theo}\label{Theorem Lower Bound curse}
        Let $r_2\in\mathbb{N}$, $t\in\mathbb{N}^{r_2}$, $d\in\mathbb{N}^{r_2+1}$, $\beta_1\geq 0$, $\beta_2\in\mathbb{R}^{r_2}$, and $B_1,B_2>0$ with  $\beta_{2,i}>0$ for $i=1,\dots,r_2$, $d_{r_2+1}=1$. Additionally, $t_1<d_1$ and $t_i\leq \min\{d_1,\dots,d_i\}$ for $i\neq 1$. Let
        \[
            \rho:=\max_{i=1,\dots,r_2}\ \frac{t_i}{\beta_{2,i}^*}
        \] 
        Define $\kappa:=1+\beta_1$ and let $\mathfrak{Q}$ be the class of all potential joint distributions $\mathbb{Q}$ of $(X,Y)$ such that (a),(b) from Corollary \ref{Corollar curse} hold for some $M> 1$, $r\in\mathbb{N}$, $\epsilon_1,\epsilon_2>0$. Let (c) hold with $c_1>0$ large enough and set
        \[
            \tau_n:=n^{\frac{1}{2\kappa+\rho-1}}.
        \]
        Then
        \begin{align*}
            \underset{n\rightarrow\infty}{\lim\ \inf}\ \underset{G_{n}\in\mathfrak{G}}{\inf}\ \underset{\mathbb{Q}\in\mathfrak{Q}}{\sup}\ \tau_n^p\mathbb{E}\left[d^p_\Delta(G_{n},G^*_\mathbb{Q})\right] & >0,\\
            \underset{n\rightarrow\infty}{\lim\ \inf}\ \underset{G_{n}\in\mathfrak{G}}{\inf}\ \underset{\mathbb{Q}\in\mathfrak{Q}}{\sup}\ \tau_n^{p\kappa}\mathbb{E}\left[d^p_{f_\mathbb{Q}}(G_{n},G^*_\mathbb{Q})\right] & >0
        \end{align*}
        for every $p\geq 0$, where $\mathfrak{G}$ contains all estimators depending on the data $(X_1,Y_1),\dots,(X_n,Y_n)$.
    \end{theo}
    
\section{Concluding Remarks}\label{Concluding Remarks}

    We establish optimal convergence rates up to a $\mathrm{log}$-factor in a classification setting under the \eqref{Characteristic Condition} using neural networks. Theorem \ref{Theorem Main} can be applied for many different boundary functions. The complexity of the class of boundary functions $\mathcal{F}$ is one of the main driving factors of the convergence rate. In particular, many approaches which circumvent the curse of dimensionality in a regression setting can be used to circumvent the curse in this classification setting.\newline
    Note that this paper is of a theoretical nature. While sparsity constraints are considered thoroughly in the theoretical literature, they are not widely used in practice. Additionally, we did not discuss the minimization required for the calculation of $\hat{G}_n$. This is a very interesting but complicated topic which is not in the scope of this article. Observe that the class of neural networks used in Theorem \ref{Theorem Main} depends on $\kappa$ as well as $\rho$. We believe that one can extend the results of this paper by either having adaptive classes of neural networks or a class independent of $\kappa$ and $\rho$ in a similar manner to \cite{tsybakov2004optimal}. One obstacle to overcome is the fact that the conditions on the probability distribution $\mathfrak{Q}$ required are not strictly weaker for larger $\kappa$ and $\rho$.\newline
    Lastly, while the goal of this paper is to prove results considering neural networks, it also contains new insights on the noise condition \eqref{Characteristic Condition}. In order to establish optimal convergence, an additional condition in order to show approximation results of neural networks with respect to the metric $d_{f_\mathbb{Q}}$ is necessary. Intuitively, the reverse inequality is required for certain sets. Note that requiring the reverse inequality is an overly restrictive assumption which holds for almost no classes of possible distributions $\mathfrak{Q}$ for $\kappa\neq 1$. This proved to be a major challenge and is solved by (3.) in Definition \ref{Definition Sets}. While this condition is always also satisfied for larger but not lower $\beta$ (and thus $\kappa$), the reverse is true in condition \eqref{Characteristic Condition}. Thus, together the requirement is that $\kappa$ is the "correct rate". Note that condition (3.) still allows for highly non-continuous $f_\mathbb{Q}$ close to the boundary of $G^*_\mathbb{Q}$. This is essential, since considering only smooth $f_\mathbb{Q}$ close to the boundary leads to different convergence rates as shown in Theorem 2 of \cite{kim2021fast}.

\appendix

\section{General Convergence Results}\label{Appendix - General Convergence Results}

    The proof of Proposition \ref{Theorem Mammen} is similar to the proof of Theorem 2 in \cite{mammen1999smooth}. For the sake of completion, we provide the entire argumentation here anyway. 
    \begin{proof}[Proof of Proposition \ref{Theorem Mammen}]
        Let $n\geq N_0$. Without loss of generality, we may assume that $\tau_n\leq n^{\frac{1}{\rho+2\kappa-1}}$, since otherwise the conditions are also satisfied when using  $\bar{\tau}_n=n^{\frac{1}{\rho+2\kappa-1}}$. We begin by proving the assertion for the first term. The idea is to bound 
        \[
            \mathbb{P}\big(d_{f_\mathbb{Q}}(\hat{G}_n,G^*_\mathbb{Q})>t\tau_n^{-\kappa}\big)
        \] 
        for some $t>0$. First, observe that for any $G\in\mathcal{N}_n$
        \begin{align*}
            & R_n(G)-R_n(G^*_\mathbb{Q})-d_{f_\mathbb{Q}}(G,G^*_\mathbb{Q})\\
            & =\frac{1}{n}\sum_{i=1}^n\big(Y_i-1(X_i\in G)\big)^2-\frac{1}{n}\sum_{i=1}^n\big(Y_i-1(X_i\in G^*_\mathbb{Q})\big)^2\\
            & \ \ \ -\bigg(\mathbb{E}\Big[\big(Y-1(X\in G)\big)^2\Big]-\mathbb{E}\Big[\big(Y-1(X\in G_\mathbb{Q}^*)\big)^2\Big]\bigg)\\
            & =\frac{1}{n}\sum_{i=1}^nh_G(X_i,Y_i)-\mathbb{E}\big[h_G(X_i,Y_i)\big]=:\frac{1}{n}\sum_{i=1}^nU_i(G)
        \end{align*}
        holds, where
        \begin{align*}
            h_G:\mathbb{R}^d\times\{0,1\}\rightarrow\mathbb{R},\ h_G(x,y)=\big(y-1(x\in G)\big)^2-\big(y-1(x\in G^*_\mathbb{Q})\big)^2.
        \end{align*}
        Regarding (iii), for every $n\in\mathbb{N}$ there exists a $G_n\in\mathcal{N}_n$ such that
        \[
            d_{f_\mathbb{Q}}(G_n,G^*_\mathbb{Q})\leq c_2\tau_n^{-\kappa}.
        \]
        For $t>0$, define
        \[
            \Xi_t:=\Big\{G\in\mathcal{N}_n\ \Big|\ d_{f_\mathbb{Q}}(G,G^*_\mathbb{Q})\geq t\tau_n^{-\kappa}\Big\}. 
        \]
        Then, for $t\geq 4c_2$ and $G\in\Xi_t$ we have
        \begin{align}
            \frac{1}{2}d_{f_\mathbb{Q}}(G,G^*_\mathbb{Q})-d_{f_\mathbb{Q}}(G_n,G^*_\mathbb{Q})\geq c_2\tau_n^{-\kappa}.\label{equation1}
        \end{align}
        Recall that by definition $\hat{G}_n$ minimizes $R_n(\cdot)$. Therefore, in view of the calculations above, for $t\geq4c_2$
        \begin{align*}
            & \mathbb{P}\big(d_{f_\mathbb{Q}}(\hat{G}_n,G^*_\mathbb{Q})>t\tau_n^{-\kappa}\big)\\ 
            &\leq \mathbb{P}\big(\exists G\in\Xi_t:\ R_n(G)-R_n(G_n)\leq 0\big)\\
            &=\mathbb{P}\Big(\exists G\in\Xi_t:\ R_n(G)-R_n(G^*_\mathbb{Q})-\big(R_n(G_n)-R_n(G^*_\mathbb{Q}\big)\leq 0\Big)\\
            & =\mathbb{P}\Bigg(\exists G\in\Xi_t:\ d_{f_\mathbb{Q}}(G,G^*_\mathbb{Q})+\frac{1}{n}\sum_{i=1}^nU_i(G)\\
            & \ \ \ \ \ \ \ \ \ \ \ \ \ \ \ \ \ \ \ \ \ \ \ \ -d_{f_\mathbb{Q}}(G_n,G^*_\mathbb{Q})-\frac{1}{n}\sum_{i=1}^nU_i(G_n)\leq 0\Bigg)
        \end{align*}
        holds. Using inequality (\ref{equation1}) in the third row yields  
        \begin{align*}
            & \mathbb{P}\Bigg(\exists G\in\Xi_t:\ d_{f_\mathbb{Q}}(G,G^*_\mathbb{Q})+\frac{1}{n}\sum_{i=1}^nU_i(G)\\
            & \ \ \ \ \ \ \ \ \ \ \ \ \ \ \ \ \ \ \ \ \ \ \ \ -d_{f_\mathbb{Q}}(G_n,G^*_\mathbb{Q})-\frac{1}{n}\sum_{i=1}^nU_i(G_n)\leq 0\Bigg)\\
            & =\mathbb{P}\Bigg(\exists G\in\Xi_t:\ \bigg(\frac{1}{2}d_{f_\mathbb{Q}}(G,G^*_\mathbb{Q})+\frac{1}{n}\sum_{i=1}^nU_i(G)\bigg)\\
            & \ \ \ \ \ \ \ \ \ \ \ \ \ \ \ \ \ \ \ \ \ \ \ \  +\bigg(\frac{1}{2}d_{f_\mathbb{Q}}(G,G^*_\mathbb{Q})-d_{f_\mathbb{Q}}(G_n,G^*_\mathbb{Q})-\frac{1}{n}\sum_{i=1}^nU_i(G_n)\bigg)\leq 0\Bigg)\\
            & \leq \mathbb{P}\Bigg(\exists G\in\Xi_t:\ \frac{1}{2}d_{f_\mathbb{Q}}(G,G^*_\mathbb{Q})+\frac{1}{n}\sum_{i=1}^nU_i(G)\leq 0\Bigg)\\
            & \ \ \ \ \ \ \ \ \ \ \ \ \ \ \ \ \ \ \ \ \ \ \ \ + \mathbb{P}\Bigg(c_2\tau_n^{-\kappa}-\frac{1}{n}\sum_{i=1}^nU_i(G_n)\leq 0\Bigg)\\
            & \leq \mathbb{P}\Bigg(\exists G\in\Xi_t:\ \frac{1}{n}\sum_{i=1}^nU_i(G)\leq -\frac{1}{2}d_{f_\mathbb{Q}}(G,G^*_\mathbb{Q})\Bigg)\\
            & \ \ \ \ \ \ \ \ \ \ \ \ \ \ \ \ \ \ \ \ \ \ \ \ + \mathbb{P}\Bigg(c_2\tau_n^{-\kappa}\leq \frac{1}{n}\sum_{i=1}^nU_i(G_n)\Bigg)
        \end{align*}
        and thus
        \begin{align*}
            & \mathbb{P}\big(d_{f_\mathbb{Q}}(\hat{G}_n,G^*_\mathbb{Q})>t\tau_n^{-\kappa}\big)\\
            & \leq\mathbb{P}\Bigg(\exists G\in\Xi_t:\ \frac{1}{n}\sum_{i=1}^nU_i(G)\leq -\frac{1}{2}d_{f_\mathbb{Q}}(G,G^*_\mathbb{Q})\Bigg)\\
            & \ \ \ \ \ \ \ \ \ \ \ \ \ \ \ \ \ \ \ \ \ \ \ \ + \mathbb{P}\Bigg(c_2\tau_n^{-\kappa}\leq \frac{1}{n}\sum_{i=1}^nU_i(G_n)\Bigg).
        \end{align*}
        It remains to find upper bounds for the two terms above. In order to bound the first term, 
        note that for $(x,y)\in\mathbb{R}^d\times\{0,1\}$ and any $G\in\mathcal{N}_n$ we have
        \begin{align*}
            |h_G(x,y)|
            & =\begin{cases}
                \big|1-1(x\in G)-\big(1-1(x\in G_\mathbb{Q}^*)\big)\big|, \ & \text{for}\ y=1,\\
                \big|1(x\in G)-1(x\in G_\mathbb{Q}^*)\big|,\ & \text{for}\ y=0
            \end{cases}\\
            & =1\big(x\in G\Delta G^*_\mathbb{Q}\big).
        \end{align*}
        For all $i=1,\dots, n$ this implies $|U_i(G)|\leq 2$ and 
        \begin{align*}
            \mathbb{E}\big[U_i(G)^2\big] & \leq\mathbb{E}\big[h_G(X_i,Y_i)^2\big] = \mathbb{E}\big[1\big(x\in G\Delta G_\mathbb{Q}^*\big)\big]\\
            & =d_\Delta(G,G^*_\mathbb{Q})\leq c^{-\frac{1}{\kappa}}_1d_{f_\mathbb{Q}}(G,G^*_\mathbb{Q})^\frac{1}{\kappa}
        \end{align*}
        where the last inequality follows from (ii). By Bernstein's inequality, for all $a>0$
        \[
            \mathbb{P}\Bigg(\bigg|\frac{1}{n}\sum_{i=1}^nU_i(G)\bigg|\geq a\Bigg)\leq 2\exp\bigg(-\frac{k_1na^2}{a+c_1^{-\frac{1}{\kappa}}d_{f_\mathbb{Q}}(G,G^*_\mathbb{Q})^\frac{1}{\kappa}}\bigg)
        \]
        holds, where $k_1>0$ is a constant. By setting $a=\frac{1}{2}d_{f_\mathbb{Q}}(G,G^*_\mathbb{Q})$ and observing that $d_{f_\mathbb{Q}}(G,G^*_\mathbb{Q})\leq 1$, we have
        \[
             \mathbb{P}\Bigg(\bigg|\frac{1}{n}\sum_{i=1}^nU_i(G)\bigg|\geq \frac{1}{2}d_{f_\mathbb{Q}}(G,G^*_\mathbb{Q})\Bigg)\leq 2\exp\Big(-k_2nd_{f,g}(G,G^*_\mathbb{Q})^{\frac{2\kappa-1}{\kappa}}\Big)
        \]
        for some constant $k_2>0$. Noting that by definition $\tau_n\leq n^{\frac{1}{\rho+2\kappa-1}}$ and $\kappa\geq1$, by (iv) we have 
        \begin{align*}
            & \mathbb{P}\Bigg(\exists G\in\Xi_t:\ \bigg|\frac{1}{n}\sum_{i=1}^nU_i(G)\bigg|\geq\frac{1}{2}d_{f_\mathbb{Q}}(G,G^*_\mathbb{Q})\Bigg)\\
            & \ \ \ \leq 2\exp\big(c_3n^{\frac{\rho}{\rho+2\kappa-1}}\big)\exp\big(-k_2nt^{\frac{2\kappa-1}{\kappa}}\tau_n^{1-2\kappa}\big)\\
            & \ \ \ \leq 2\exp\big(c_3n^{\frac{\rho}{\rho+2\kappa-1}}\big)\exp\big(-k_2t^{\frac{2\kappa-1}{\kappa}}n^{\frac{1-2\kappa}{\rho+2\kappa-1}+1}\big)\\
            & \ \ \ \leq 2\exp\Big(\big(c_3-k_2t^{\frac{2\kappa-1}{\kappa}}\big)n^{\frac{\rho}{\rho+2\kappa-1}}\Big)\\
            & \ \ \ \leq 2\exp\Big(-c_3\tau_n^\rho\Big)
        \end{align*}
        for all $t\geq \Big(\frac{2c_3}{k_2}\Big)^{\frac{\kappa}{2\kappa-1}}$. To bound the second term we use Bernstein's inequality with $a=c_2\tau_n^{-\kappa}$ and receive
        \begin{align*}
            \mathbb{P}\Bigg(c_2\tau_n^{-\kappa}\leq \frac{1}{n}\sum_{i=1}^nU_i(G_n)\Bigg) & \leq \exp\bigg(-\frac{k_1nc_2^2\tau_n^{-2\kappa}}{c_2\tau_n^{-\kappa}+c_1^{-\frac{1}{\kappa}}d_{f_\mathbb{Q}}(G^*_\mathbb{Q},G_n)^{\frac{1}{\kappa}}}\bigg)\\
            & \leq \exp\big(-k_3n\tau_n^{-2\kappa+1}\big)\\
            &\leq\exp\big(-k_3\tau_n^\rho\big)
        \end{align*}
        for some constant $k_3>0$. Therefore, for $t\geq\max\Big\{4c_2,\Big(\frac{2c_3}{k_2}\Big)^{\frac{\kappa}{2\kappa-1}}\Big\}$ we find an upper bound
        \begin{align*}
            & \mathbb{P}\big(d_{f_\mathbb{Q}}(\hat{G}_n,G^*_\mathbb{Q})>t\tau_n^{-\kappa}\big)\\ 
            &\leq \mathbb{P}\Bigg(\exists G\in\Xi_t:\ \frac{1}{n}\sum_{i=1}^nU_i(G)\leq -\frac{1}{2}d_{f_\mathbb{Q}}(G,G^*_\mathbb{Q})\Bigg)\\
            & \ \ \ \ \ \ \ \ \ \ \ \ \ \ \ \ \ \ \ \ \ \ \ \ + \mathbb{P}\Bigg(c_2\tau_n^{-\kappa}\leq \frac{1}{n}\sum_{i=1}^nU_i(G_n)\Bigg)\\
            &\leq 2\exp\Big(-c_3\tau_n^\rho\Big)+ \exp\big(-k_3\tau_n^\rho\big).
        \end{align*}
        Observing that $d_{f_\mathbb{Q}}(\hat{G}_n,G^*_\mathbb{Q})\leq 1$ we conclude
        \begin{align*}
            & \mathbb{E}\big[d_{f_\mathbb{Q}}^p(\hat{G}_n,G^*_\mathbb{Q})\big]\\
            & \leq \mathbb{E}\big[1\big(d_{f_\mathbb{Q}}(\hat{G}_n,G^*_\mathbb{Q})>t\tau_n^{-\kappa}\big)\big]+t\tau_n^{-p\kappa}\mathbb{E}\big[1\big(d_{f_\mathbb{Q}}(\hat{G}_n,G^*_\mathbb{Q})\leq t\tau_n^{-\kappa}\big)\big]\\
            & \leq 2\exp\big(-c_3\tau_n^\rho\big)+ \exp\big(-k_3\tau_n^\rho\big)+ t\tau_n^{-\kappa p}
        \end{align*}
        and thus
        \begin{align*}
            & \underset{n\rightarrow\infty}{\lim\ \sup}\ \underset{\mathbb{Q}\in\mathfrak{Q}}{\sup}\ \tau_n^{\kappa p}\mathbb{E}\big[d_{f_\mathbb{Q}}^p(\hat{G}_n,G^*_\mathbb{Q})\big]\\
            & \leq \underset{n\rightarrow\infty}{\lim\ \sup}\ \tau_n^{\kappa p}\Big(2\exp\big(-c_3\tau_n^\rho\big)+ \exp\big(-k_3\tau_n^\rho\big)+ t\tau_n^{-\kappa p}\Big)\\
            & <\infty.
        \end{align*}
        Proving that the second term in the assertion is finite follows directly, since regarding (ii) for all $\mathbb{Q}\in\mathfrak{Q}$ and sets $G\in\mathcal{N}_n$ it holds hat 
        \[
            d_{f_\mathbb{Q}}(G,G^*_\mathbb{Q})\geq c_1d^\kappa_\Delta(G,G^*_\mathbb{Q}).
        \]
    \end{proof}

    \begin{proof}[Proof of Proposition \ref{Theorem Mammen 2}]
        Let $n\geq N_0$. Without loss of generality, we may assume that $\tau_n\leq n^{\frac{1}{\rho+2\kappa-1}}$, since otherwise the conditions are also satisfied when using $\tau_n=n^{\frac{1}{\rho+2\kappa-1}}$. The idea is to bound 
        \[
            \mathbb{P}\big(d_{f_\mathbb{Q}}(\hat{G}_n,G^*_\mathbb{Q})>t\tau_n^{-1}\big)
        \] 
        for some $t>0$. First, observe that for any $G\in\mathcal{N}_n$
        \begin{align*}
            & R_n(G)-R_n(G^*_\mathbb{Q})-d_{f_\mathbb{Q}}(G,G^*_\mathbb{Q})\\
            & =\frac{1}{n}\sum_{i=1}^n\big(Y_i-1(X_i\in G)\big)^2-\frac{1}{n}\sum_{i=1}^n\big(Y_i-1(X_i\in G^*_\mathbb{Q})\big)^2\\
            & \ \ \ -\bigg(\mathbb{E}\Big[\big(Y-1(X\in G)\big)^2\Big]-\mathbb{E}\Big[\big(Y-1(X\in G_\mathbb{Q}^*)\big)^2\Big]\bigg)\\
            & =\frac{1}{n}\sum_{i=1}^nh_G(X_i,Y_i)-\mathbb{E}\big[h_G(X_i,Y_i)\big]=:\frac{1}{n}\sum_{i=1}^nU_i(G)
        \end{align*}
        holds, where
        \begin{align*}
            h_G:\mathbb{R}^d\times\{0,1\}\rightarrow\mathbb{R},\ h_G(x,y)=\big(y-1(x\in G)\big)^2-\big(y-1(x\in G^*_\mathbb{Q})\big)^2.
        \end{align*}
        Regarding (ii), for every $n\in\mathbb{N}$ there exists a $G_n\in\mathcal{N}_n$ such that
        \[
            d_{f_\mathbb{Q}}(G_n,G^*_\mathbb{Q})\leq c_2\tau_n^{-1}.
        \]
        For $t>0$, define
        \[
            \Xi_t:=\Big\{G\in\mathcal{N}_n\ \Big|\ d_{f_\mathbb{Q}}(G,G^*_\mathbb{Q})\geq t\tau_n^{-1}\Big\}. 
        \]
        Then, for $t\geq 4c_2$ and $G\in\Xi_t$ we have
        \begin{align}
            \frac{1}{2}d_{f_\mathbb{Q}}(G,G^*_\mathbb{Q})-d_{f_\mathbb{Q}}(G_n,G^*_\mathbb{Q})\geq c_2\tau_n^{-1}.\label{equation2}
        \end{align}
        Recall that by definition $\hat{G}_n$ minimizes $R_n(\cdot)$. Therefore, in view of the calculations above, for $t\geq4c_2$
        \begin{align*}
            & \mathbb{P}\big(d_{f_\mathbb{Q}}(\hat{G}_n,G^*_\mathbb{Q})>t\tau_n^{-1}\big)\\ 
            &\leq \mathbb{P}\big(\exists G\in\Xi_t:\ R_n(G)-R_n(G_n)\leq 0\big)\\
            &=\mathbb{P}\Big(\exists G\in\Xi_t:\ R_n(G)-R_n(G^*_\mathbb{Q})-\big(R_n(G_n)-R_n(G^*_\mathbb{Q}\big)\leq 0\Big)\\
            & =\mathbb{P}\Bigg(\exists G\in\Xi_t:\ d_{f_\mathbb{Q}}(G,G^*_\mathbb{Q})+\frac{1}{n}\sum_{i=1}^nU_i(G)\\
            & \ \ \ \ \ \ \ \ \ \ \ \ \ \ \ \ \ \ \ \ \ \ \ \ -d_{f_\mathbb{Q}}(G_n,G^*_\mathbb{Q})-\frac{1}{n}\sum_{i=1}^nU_i(G_n)\leq 0\Bigg)
        \end{align*}
        holds. Using inequality (\ref{equation2}) in the third row yields  
        \begin{align*}
            & \mathbb{P}\Bigg(\exists G\in\Xi_t:\ d_{f_\mathbb{Q}}(G,G^*_\mathbb{Q})+\frac{1}{n}\sum_{i=1}^nU_i(G)\\
            & \ \ \ \ \ \ \ \ \ \ \ \ \ \ \ \ \ \ \ \ \ \ \ \ -d_{f_\mathbb{Q}}(G_n,G^*_\mathbb{Q})-\frac{1}{n}\sum_{i=1}^nU_i(G_n)\leq 0\Bigg)\\
            & =\mathbb{P}\Bigg(\exists G\in\Xi_t:\ \bigg(\frac{1}{2}d_{f_\mathbb{Q}}(G,G^*_\mathbb{Q})+\frac{1}{n}\sum_{i=1}^nU_i(G)\bigg)\\
            & \ \ \ \ \ \ \ \ \ \ \ \ \ \ \ \ \ \ \ \ \ \ \ \  +\bigg(\frac{1}{2}d_{f_\mathbb{Q}}(G,G^*_\mathbb{Q})-d_{f_\mathbb{Q}}(G_n,G^*_\mathbb{Q})-\frac{1}{n}\sum_{i=1}^nU_i(G_n)\bigg)\leq 0\Bigg)\\
            & \leq \mathbb{P}\Bigg(\exists G\in\Xi_t:\ \frac{1}{2}d_{f_\mathbb{Q}}(G,G^*_\mathbb{Q})+\frac{1}{n}\sum_{i=1}^nU_i(G)\leq 0\Bigg)\\
            & \ \ \ \ \ \ \ \ \ \ \ \ \ \ \ \ \ \ \ \ \ \ \ \ + \mathbb{P}\Bigg(c_2\tau_n^{-1}-\frac{1}{n}\sum_{i=1}^nU_i(G_n)\leq 0\Bigg)\\
            & \leq \mathbb{P}\Bigg(\exists G\in\Xi_t:\ \frac{1}{n}\sum_{i=1}^nU_i(G)\leq -\frac{1}{2}d_{f_\mathbb{Q}}(G,G^*_\mathbb{Q})\Bigg)\\
            & \ \ \ \ \ \ \ \ \ \ \ \ \ \ \ \ \ \ \ \ \ \ \ \ + \mathbb{P}\Bigg(c_2\tau_n^{-1}\leq \frac{1}{n}\sum_{i=1}^nU_i(G_n)\Bigg)
        \end{align*}
        and thus
        \begin{align*}
            & \mathbb{P}\big(d_{f_\mathbb{Q}}(\hat{G}_n,G^*_\mathbb{Q})>t\tau_n^{-1}\big)\\
            & \leq\mathbb{P}\Bigg(\exists G\in\Xi_t:\ \frac{1}{n}\sum_{i=1}^nU_i(G)\leq -\frac{1}{2}d_{f_\mathbb{Q}}(G,G^*_\mathbb{Q})\Bigg)\\
            & \ \ \ \ \ \ \ \ \ \ \ \ \ \ \ \ \ \ \ \ \ \ \ \ + \mathbb{P}\Bigg(c_2\tau_n^{-1}\leq \frac{1}{n}\sum_{i=1}^nU_i(G_n)\Bigg)
        \end{align*}
        It remains to find upper bounds for the two terms above. In order to bound the first term, 
        note that for $(x,y)\in\mathbb{R}^d\times\{0,1\}$ and any $G\in\mathcal{N}_n$ we have
        \begin{align*}
            |h_G(x,y)|
            & =\begin{cases}
                \big|1-1(x\in G)-\big(1-1(x\in G_\mathbb{Q}^*)\big)\big|, \ & \text{for}\ y=1,\\
                \big|1(x\in G)-1(x\in G_\mathbb{Q}^*)\big|,\ & \text{for}\ y=0
            \end{cases}\\
            & =\begin{cases}
                \big|1(x\in G^*_\mathbb{Q})-1(x\in G)\big|, \ & \text{for}\ y=1,\\
                \big|1(x\in G)-1(x\in G_\mathbb{Q}^*)\big|,\ & \text{for}\ y=0
            \end{cases}\\
            & =1\big(x\in G\Delta G^*_\mathbb{Q}\big).
        \end{align*}
        For all $i=1,\dots, n$ this implies $|U_i(G)|\leq 2$ and 
        \[
            \mathbb{E}\big[U_i(G)^2\big]\leq\mathbb{E}\big[h_G(X_i,Y_i)^2\big]= \mathbb{E}\big[1\big(x\in G\Delta G_\mathbb{Q}^*\big)\big]=d_\Delta(G,G^*_\mathbb{Q})\leq 1.
        \] 
        By Bernstein's inequality, for all $a>0$
        \[
            \mathbb{P}\Bigg(\bigg|\frac{1}{n}\sum_{i=1}^nU_i(G)\bigg|\geq a\Bigg)\leq 2\exp\bigg(-\frac{k_1na^2}{a+1}\bigg)
        \]
        holds, where $k_1>0$ is a constant. By setting $a=\frac{1}{2}d_{f_\mathbb{Q}}(G,G^*_\mathbb{Q})$, we have
        \[
             \mathbb{P}\Bigg(\bigg|\frac{1}{n}\sum_{i=1}^nU_i(G)\bigg|\geq \frac{1}{2}d_{f_\mathbb{Q}}(G,G^*_\mathbb{Q})\Bigg)\leq 2\exp\Big(-k_2nd_{f,g}(G,G^*_\mathbb{Q})^{2}\Big)
        \]
        for some constant $k_2>0$. Noting that by definition $\tau_n\leq n^{\frac{1}{\rho+2}}$, by (iii) we have 
        \begin{align*}
            & \mathbb{P}\Bigg(\exists G\in\Xi_t:\ \bigg|\frac{1}{n}\sum_{i=1}^nU_i(G)\bigg|\geq\frac{1}{2}d_{f_\mathbb{Q}}(G,G^*_\mathbb{Q})\Bigg)\\
            & \ \ \ \leq 2\exp\big(c_3n^{\frac{\rho}{\rho+2}}\big)\exp\big(-k_2nt^{2}\tau_n^{-2}\big)\\
            & \ \ \ \leq 2\exp\big(c_3n^{\frac{\rho}{\rho+2}}\big)\exp\big(-k_2t^{2}n^{-\frac{2}{\rho+2}+1}\big)\\
            & \ \ \ \leq 2\exp\Big(\big(c_3-k_2t^{2}\big)n^{\frac{\rho}{\rho+2}}\Big)\\
            & \ \ \ \leq 2\exp\Big(-c_3\tau_n^\rho\Big)
        \end{align*}
        for all $t\geq \sqrt{\frac{2c_3}{k_2}}$. To bound the second term we use Bernstein's inequality with $a=c_2\tau_n^{-1}$ and receive
        \begin{align*}
            \mathbb{P}\Bigg(c_2\tau_n^{-1}\leq \frac{1}{n}\sum_{i=1}^nU_i(G_n)\Bigg) & \leq \exp\bigg(-\frac{k_1nc_2^2\tau_n^{-2}}{c_2\tau_n^{-1}+1}\bigg)\\
            & \leq \exp\big(-k_3n\tau_n^{-2}\big)\\
            &\leq\exp\big(-k_3\tau_n^\rho\big)
        \end{align*}
        for a constant $k_3>0$. Therefore, for $t\geq\max\Big\{4c_2,\sqrt{\frac{2c_3}{k_2}}\Big\}$ we find an upper bound
        \begin{align*}
            & \mathbb{P}\big(d_{f_\mathbb{Q}}(\hat{G}_n,G^*_\mathbb{Q})>t\tau_n^{-1}\big)\\ 
            &\leq \mathbb{P}\Bigg(\exists G\in\Xi_t:\ \frac{1}{n}\sum_{i=1}^nU_i(G)\leq -\frac{1}{2}d_{f_\mathbb{Q}}(G,G^*_\mathbb{Q})\Bigg)\\
            & \ \ \ \ \ \ \ \ \ \ \ \ \ \ \ \ \ \ \ \ \ \ \ \ + \mathbb{P}\Bigg(c_2\tau_n^{-1}\leq \frac{1}{n}\sum_{i=1}^nU_i(G_n)\Bigg)\\
            &\leq 2\exp\Big(-c_3\tau_n^\rho\Big)+ \exp\big(-k_3\tau_n^\rho\big).
        \end{align*}
        Observing that $d_{f_\mathbb{Q}}(\hat{G}_n,G^*_\mathbb{Q})\leq 1$, we conclude
        \begin{align*}
            & \mathbb{E}\big[d_{f_\mathbb{Q}}^p(\hat{G}_n,G^*_\mathbb{Q})\big]\\
            & \leq \mathbb{E}\big[1\big(d_{f_\mathbb{Q}}(\hat{G}_n,G^*_\mathbb{Q})>t\tau_n^{-1}\big)\big]+t\tau_n^{-p}\mathbb{E}\big[1\big(d_{f_\mathbb{Q}}(\hat{G}_n,G^*_\mathbb{Q})\leq t\tau_n^{-1}\big)\big]\\
            & \leq 2\exp\Big(-c_3\tau_n^\rho\Big)+ \exp\big(-k_3\tau_n^\rho\big)+ t\tau_n^{-p}
        \end{align*}
        and thus
        \begin{align*}
            & \underset{n\rightarrow\infty}{\lim\ \sup}\ \underset{f\in\mathcal{F}}{\sup}\ \tau_n^{ p}\mathbb{E}\big[d_{f_\mathbb{Q}}^p(\hat{G}_n,G^*_\mathbb{Q})\big]\\
            & \leq \underset{n\rightarrow\infty}{\lim\ \sup}\ \tau_n^{ p}\Big(2\exp\Big(-c_3\tau_n^\rho\Big)+ \exp\big(-k_3\tau_n^\rho\big)+ t\tau_n^{- p}\Big)\\
            & <\infty.
        \end{align*}
    \end{proof}

\section{Convergence Rates for Neural Networks}\label{Appendix - Convergence Rates for Neural Networks}

    The first goal of this section is to prove Theorem \ref{Theorem Main}. We then follow this up by proving Lemma \ref{Lemma Theorem 5} and Lemma \ref{Lemma curse}.

\subsection{Proof of the Main Result}

    \noindent In order to simplify the approximation results below, we introduce a lemma considering the parallelization and concatenation of two networks $\Phi_1$ and $\Phi_2$. Since these results have been shown in many other articles e.g.  \cite{petersen2018optimal, schmidt2020nonparametric}, we omit the proof.   
    
    \begin{lemma}\label{lemma para+concat}
        Let $R(\Phi_1):\mathbb{R}^{d_1}\rightarrow\mathbb{R}^{d_3}$ and $R(\Phi_2):\mathbb{R}^{d_2}\rightarrow\mathbb{R}^{d_4}$ be realizations of neural networks with $L_1,L_2$ layers, sparsity $s_1,s_2$ and weights in $\mathcal{W}_{c_1},\mathcal{W}_{c_2}$, respectively.  
        \begin{itemize}
            \item If $d_4=d_1$, the concatenation of the functions $R(\Phi_1)\circ R(\Phi_2)$ can be realized by a neural network with $L=L_1+L_2+1$ layers, sparsity $s\leq 2s_1+2s_2$ and weights in $\mathcal{W}_{\max\{c_1,c_2\}}$.
            \item If $d_1=d_2$, the parallelization of the functions $P(R(\Phi_1),R(\Phi_2)):\mathbb{R}^{d_1}\rightarrow\mathbb{R}^{d_3+d_4}$ given by
            \[
                P(R(\Phi_1),R(\Phi_2))(x):=(R(\Phi_1)(x),R(\Phi_2)(x))
            \]
            can be realized by a neural network with $L=\max\{L_1,L_2\}$ layers, sparsity $s\leq s_1+s_2+2dL$ and weights in $\mathcal{W}_{\max\{c_1,c_2\}}$.
        \end{itemize}
    \end{lemma}
    
    Note that we are only using weights $|w|\leq 1$. In order to approximate high numbers, we use the following lemma.

\begin{lemma}\label{Lemma high numbers}
        Let $c,M\in\mathbb{N}$. Then there exist neural networks $\Phi_1,\Phi_2$ with input dimensions $m_0=1$, at most $L=M+1$ layers, sparsity $s\leq 4M+1$ and weights in $\mathcal{W}_{c}$ such that
        \begin{align*}
            R(\Phi_1)(x) & = 2^Mx,\\
            R(\Phi_2)(x) & = 2^M
        \end{align*}
        for all $x\in[0,1]$.
    \end{lemma}
    
    \begin{proof}
        The network $\Phi_1$ is given by
        \[
            \Phi_1:=\big(W_1,b_1,\dots,W_{M+1},b_{M+1}\big)
        \]
        where $W_{M+1}=W_1^T= (1\ 1)$, 
        \begin{align*}
            W_i=\begin{pmatrix}
                1 & 1 \\
                1 & 1 \\
            \end{pmatrix}
        \end{align*}
        for $i=2,\dots, M$, $b_i=(0, 0)$ for $i\leq M$ and $b_{M+1}=0$. The other network is
        \[
            \Phi_2 := (W_0,b_0,\dots,W_{M+1},b_{M+1})=(W_0,b_0)\times \Phi_1
        \]
        where $W_0=0$ and $b_0=1$.\newline
        
        \noindent The layers of both networks are bounded by $M+1$. Sparsity of both networks can be bounded by
        \[
            s\leq 2*2+4(M-1)+1=4M+1.
        \]
    \end{proof}

    Next, we construct a neural network for each $G\in\mathcal{K}_{\mathbb{Q},\beta,B,\epsilon_1,\epsilon_2,r,d}^\mathcal{F}$ which approximates $G$ well with respect to the metric $d_{f_\mathbb{Q}}$. The rough idea for the construction of the network is similar to ideas used in  \cite{petersen2018optimal}. However, the precise construction in order to adapt to the metric in question differs substantially. The proof of the following theorem is one of the main contributions of this paper.  

    \begin{theo}\label{Theorem Approximation Result}
        Let $\beta\geq 0$, $B,\rho>0$ and $d\in\mathbb{N}$ with $d\geq 2$. Let $\mathcal{F}$ be a set of functions 
        \[
            \gamma:[0,1]^{d-1}\rightarrow\mathbb{R}
        \]
        such that the following holds. There exist $\epsilon_0,C_1,C_2>0$ and $C_3,C_4\in\mathbb{N}$ such that for any $\gamma\in \mathcal{F}$ and any $\epsilon\in(0,\epsilon_0)$ there is a neural network $\Phi$ with $L\leq L_0(\epsilon):=C_1 \lceil\log(\epsilon^{-1})\rceil$ layers, sparsity $s\leq s_0(\epsilon):=C_2 \epsilon^{-\rho}\log(\epsilon^{-1})$ and weights in $\mathcal{W}_c$ with $c=c_0(\epsilon):=C_3+C_4\lceil\log(\epsilon^{-1})\rceil$ such that
        \[
            \|R(\Phi)(x)-f\|_\infty\leq\epsilon.
        \]
        Define $\kappa=1+\beta$ and let $\mathfrak{Q}$ be a class of potential joint distributions $\mathbb{Q}$ of $(X,Y)$ such that the following conditions hold. 
        \begin{itemize}
            \item[(a)] There is a constant $M>1$ such that for all $\mathbb{Q}\in\mathfrak{Q}$ the marginal distribution of $\mathbb{Q}_X$ has a Lebesgue density bounded by $M$.
            \item[(b)] There are constants $r\in\mathbb{N}$ and $\epsilon_1,\epsilon_2>0$ such that for all $\mathbb{Q}\in\mathfrak{Q}$ the bayes rule satisfies $G_{\mathbb{Q}}^*\in \mathcal{K}_{\mathbb{Q},\beta,B,\epsilon_1,\epsilon_2,r,d}^\mathcal{F}$.
        \end{itemize}
        Let
        \[
            \tau_n:=\frac{n^{\frac{1}{2\kappa+\rho-1}}}{\log^{\frac{2}{\rho}}(n)}.
        \]
        Then there exist constants $C_1',C_2'>0$ and $C_3'\in\mathbb{N}$  such that the set
        \[
            \mathcal{N}_n=\mathcal{N}_{C_1'L_0(\tau_n^{-1}),C_2's_0(\tau_n^{-1}),C_3'c_0(\tau_n^{-1})}
        \]
        satisfies the following property. There is a constants $c_2>0$ and $N_0\in\mathbb{N}$ such that for all $n\geq N_0$ and $\mathbb{Q}\in\mathfrak{Q}$ there is a $G\in\mathcal{N}_n$ with
        \[
            d_{f_\mathbb{Q}}(G,G^*_\mathbb{Q})\leq c_2\tau_n^{-\kappa}.
        \]
    \end{theo}
    
    \begin{proof}
        Set
        $\epsilon_0:= \min\{\epsilon_1,\frac{\epsilon_2}{4}\}$. Choose $N_0$ large enough such that $\tau_{N_0}\geq \epsilon_0^{-1}$. The proof is outlined as follows. We first construct a candidate set $G$ using neural networks. Then, we show that it satisfies the desired properties.\newline
        
        \noindent Let $n\geq N_0$, $\mathbb{Q}\in\mathfrak{Q}$ and
        \[
            G_\mathbb{Q}^*=H_1\cup\dots\cup H_u
        \]
        as in Definition \ref{Definition Sets} with $u\leq r$. We begin with the construction of the candidate set $G$. The idea is to define a network which approximates $G_\mathbb{Q}^*$ well on each set $H_\nu$ separately. Define $\iota_\nu,j_\nu,a_i^\nu,b_i^\nu,D_\nu,\gamma_\nu$ and $g_{\nu,x}$ as in Definition \ref{Definition Sets}. First, for each $\nu=1,\dots,u$ we consider a set $\tilde{D}_\nu$ with boarders lying on a grid. The advantage of using $\tilde{H}_\nu=\tilde{D}_\nu\cap H_\nu$ instead of $H_\nu$ is twofold. On the one hand, the grid and parameters of $\mathcal{N}_n$ are defined such that the boarders of $\tilde{D}_\nu$ can be constructed precisely. On the other, using the grid, two sets $\tilde{H}_{\nu_1}$, $\tilde{H}_{\nu_2}$ have a minimum distance for $\nu_1\neq\nu_2$, which is important for our method to work.  For $\delta>0$ let
        \[
            h_\delta=\max\Big\{h=2^{-c} \Big|\ h\leq\delta,\ c\in\mathbb{N}\Big\}.
        \]
        Set $\epsilon:=\tau_n^{-1}$. Define $I:=\big\{0,h_{\epsilon^\kappa},2h_{\epsilon^\kappa}\dots, 1-h_{\epsilon^\kappa}\big\}$ and let
        \begin{align*}
            \tilde{a}^\nu_j & := \min\{a\in I\ |\ a> a_j^\nu\},\\
            \tilde{b}^\nu_j & :=\min\{b\in I\ |\ b< b_j^\nu\},
        \end{align*} 
        for $\nu=1,\dots,u$, $j=1,\dots,d$. Now, set
        \[
            \tilde{D}_\nu:=\begin{cases}
                \prod_{j=1}^d \big[\tilde{a}^\nu_j,\tilde{b}^\nu_j\big],\ &\text{if}\ \forall j=1,\dots,d\  \tilde{a}_j^\nu<\tilde{b}_j,\\
                \emptyset,\ &\text{otherwise}.
            \end{cases}
        \]
        Note that $\tilde{b}^\nu_{j_\nu}-\tilde{a}^\nu_{j_\nu}\geq 2\epsilon$ for all $\nu=1,\dots,u$ by the choice of $\epsilon_0$. Figure \ref{fig:Menge_D} shows the collection of sets $\tilde{D}_\nu$ in the example considered in Figure \ref{fig:Menge}. 
        \begin{figure}
            \begin{center}
                \includegraphics[scale=0.3]{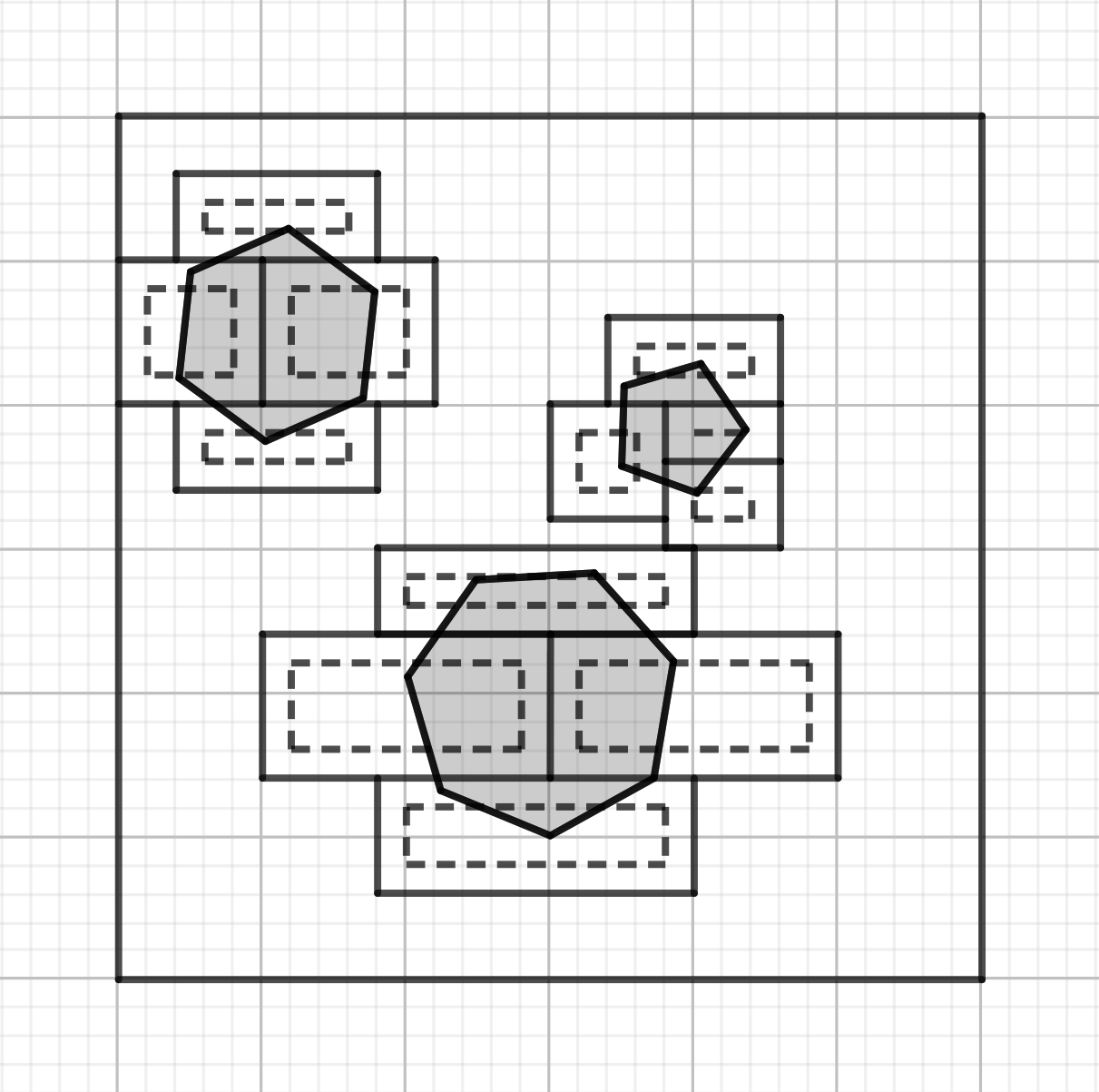}
            \end{center}
		    \caption{The collection of sets $\tilde{D}_\nu$ when considering the example from Figure \ref{fig:Menge}. The dotted lines are the boarders of the sets $D_1,\dots,D_{12}$. Note that $\delta$ is quite large in this example and observe, that the distance between two sets $\tilde{D}_{\nu_1},\tilde{D}_{\nu_2}$ is at least $2^{\delta+1}$.} \label{fig:Menge_D}
	    \end{figure}
	    Obviously we have $\tilde{D}_\nu\subseteq D_\nu$. Let
        \[
            \tilde{H}_\nu=\tilde{D}_\nu\cap\{x\in [0,1]^d\ |\ \iota_\nu x_{j_\nu}\leq \gamma_\nu(x_{-j_\nu})\}.
        \]
        The idea is to construct a neural network for every $\nu=1,\dots,u$ with $\tilde{D}_\nu\neq\emptyset$ which approximates $\tilde{H}_\nu$. We obtain the final neural network by parallelizing and adding up these networks. More specifically, we construct a network that approximates the product of $\mathbbm{1}(x \in \tilde{D}_\nu)$ and $\mathbbm{1}(\iota_\nu x_{j_\nu}\leq \gamma_\nu(x_{-j_\nu}))$. The latter is approximated by a network $\Phi_{\gamma_\nu}$ which is the concatenation of a network approximating the heaviside function $\mathbbm{1}(x_{j_\nu}>0)$ and a network approximating
        \[
            \tilde\gamma_\nu(x):=\big(x_1,\dots,x_{j_\nu-1},\iota_{\nu}x_{j_\nu}-\gamma_\nu(x_{-j_\nu}),x_{j_\nu+1},\dots,x_d\big).
        \]
        For $\nu=1,\dots,u$, from the prerequisites given in the Theorem we obtain a network $\Phi^1_{\gamma_\nu}$ with $L^1_{\gamma_\nu}\leq C_{1}^1\lceil\log\epsilon^{-1}\rceil$, $s^1_{\gamma_\nu}\leq C_{2}^1\epsilon^{-\rho}\lceil\log\epsilon^{-1}\rceil$ and weights in $\mathcal{W}_{c^1}$ with $c^1:=C_3^1+C_4^1\lceil\log\epsilon^{-1}\rceil$, such that
        \[
            \|R(\Phi^1_{\gamma_\nu})-\gamma_\nu\|_\infty\leq\frac{\epsilon}{4}. 
        \]
        For technical reasons, we need to slightly change the realisations in order to handle the behaviour of the approximations at the boarder of $\tilde{D}_\nu$. Define $\Phi^2_{\gamma_\nu}$ by its realization
        \begin{align*}
            R(\Phi^2_{\gamma_\nu})(x) & :=\frac{\tilde{b}_{j_\nu}^\nu +\tilde{a_{j_\nu}^\nu}}{2} +\sigma\left(R(\Phi^1_{\gamma_\nu})(x)+h_{\frac{\epsilon}{2}}-\frac{\tilde{b}_{j_\nu}^\nu+\tilde{a_{j_\nu}^\nu}}{2}\right)\\
            & \ \ \ \ \ \ \ \ \ \ \ \ \ \ \ \ \ -\sigma\left(\frac{\tilde{b}_{j_\nu}^\nu+\tilde{a_{j_\nu}^\nu}}{2}-R(\Phi^1_{\gamma_\nu})(x)+h_{\frac{\epsilon}{2}}\right)\\
            & = \begin{cases}
                R(\Phi^1_{\gamma_\nu})(x)+h_{\frac{\epsilon}{2}},\ & \text{if}\ R(\Phi^1_{\gamma_\nu})(x)\geq \frac{\tilde{b}_{j_\nu}^\nu +\tilde{a_{j_\nu}^\nu}}{2}+h_{\frac{\epsilon}{2}},\\
                R(\Phi^1_{\gamma_\nu})(x)-h_{\frac{\epsilon}{2}},\ & \text{if}\ R(\Phi^1_{\gamma_\nu})(x)\leq \frac{\tilde{b}_{j_\nu}^\nu +\tilde{a_{j_\nu}^\nu}}{2}-h_{\frac{\epsilon}{2}},\\
                2R(\Phi^1_{\gamma_\nu})(x)-\frac{\tilde{b}_{j_\nu}^\nu +\tilde{a_{j_\nu}^\nu}}{2},\ & \text{otherwise}.
            \end{cases}
        \end{align*}
        Let $\hat{\gamma}_\nu:=R(\Phi^2_{\gamma_\nu})$. Note that
        \begin{align*}
            \|\hat{\gamma}_\nu-\gamma_\nu\|_\infty & \leq\|R(\Phi^1_{\gamma_\nu})-\gamma_\nu\|_\infty+h_{\frac{\epsilon}{2}}\leq\frac{\epsilon}{4}+\frac{\epsilon}{2}\leq\epsilon
        \end{align*}
        as well as
        \begin{align*}
            \hat{\gamma}_\nu(x) & \geq \gamma_\nu(x)\ \text{for}\ x: \gamma_\nu(x)\geq \tilde{b}_{j_\nu}^\nu,\\
            \hat{\gamma}_\nu(x) & \leq \gamma_\nu(x)\ \text{for}\ x: \gamma_\nu(x)\leq \tilde{a}_{j_\nu}^\nu.
        \end{align*}
        Using parallelization and concatenation of Lemma \ref{lemma para+concat}, the function 
        \begin{align*}
            R(\Phi^3_{\gamma_\nu})(x):=\big(x_1,\dots,x_{j_\nu-1},\iota_{\nu}x_{j_\nu}-\hat{\gamma_\nu}(x_{-j_\nu}),x_{j_\nu+1},\dots,x_d\big)    
        \end{align*}
        is a realization of a neural network  with at most $L^3_{\gamma_\nu}\leq L^1_{\gamma_\nu}+3$ layers, sparsity 
        \[
            s^3_{\gamma_\nu}\leq 2s^1_{\gamma_\nu}+14+6d+2dL^1_{\gamma_\nu},
        \]
        and weights in $\mathcal{W}_{c^3}$ with $c^3=c^1+2$. Now let
        \[
            R(\Phi_H)(x):=\sigma(x_{j_\nu}+1)-\sigma(x_{j_\nu})=\begin{cases}
            0,\ & \text{for}\ x_{j_\nu}\leq -1,\\
            x_{j_\nu}+1,\ & \text{for}\ -1<x_{j_\nu}<0,\\
            1,\ & \text{for}\ x_{j_\nu}\geq 0.
            \end{cases}
        \]
        Note that $R(\Phi_H)(x)\in(0,1)$ if $x_{j_\nu}\in(-1,0)$. Define $\Phi_{\gamma_\nu}:=\Phi_H\circ\Phi^3_{\gamma_\nu}$ as in Lemma \ref{lemma para+concat} concatenation. The network $\Phi_{\gamma_\nu}$ has $L_{\gamma_\nu}\leq C_1^{\gamma_\nu}\lceil\log\epsilon^{-1}\rceil$ layers, sparsity $s_{\gamma_\nu}\leq C_2^{\gamma_\nu}\epsilon^{-\rho}\lceil\log(\epsilon^{-1})\rceil$ and weights in $\mathcal{W}_{c^{\gamma_\nu}}$ with $c^{\gamma_\nu}:=C_3^{\gamma_\nu}+C_4^{\gamma_\nu}\lceil\log(\epsilon^{-1})\rceil$ for some constants $C_1^{\gamma_\nu},C_2^{\gamma_\nu}>0$, $C_3^{\gamma_\nu},C_4^{\gamma_\nu}\in\mathbb{N}$. Then $R(\Phi_{\gamma_\nu})(x)=1$ if $\iota_{\nu}x_{j_\nu}\leq\hat{\gamma}_\nu(x_{-j_\nu})$ and $0\leq R(\Phi_{\gamma_\nu})(x)<1$ otherwise.
        Next, for $\nu=1,\dots,u$ with $\tilde{H}_\nu\neq\emptyset$ and $i\in\{1,\dots,d\}$, define the network $\Phi_{\nu,i}$ with realization 
        \begin{align*}
             R(\Phi_{\nu,i})(x) & :=2h_{\epsilon^{\kappa}}^{-1}\bigg(\sigma\left(x_i-\tilde{a}^\nu_i+\frac{h_{\epsilon^\kappa}}{2}\right)-\sigma\left(x_i-\tilde{a}_i^\nu\right)-\sigma\left(x_i-\tilde{b}_i^\nu\right)\\
             & \ \ \ \ \ \  +\sigma\left(x_i-\tilde{b}_i^\nu-\frac{h_{\epsilon^\kappa}}{2}\right)\bigg)\\
             & =\begin{cases}
            0,\ & \text{for}\ x_i\leq \tilde{a}^\nu_i-\frac{h_{\epsilon^\kappa}}{2},\\
            2h_{\epsilon^{\kappa}}^{-1}\left(x_i-\tilde{a}^\nu_i+\frac{h_{\epsilon^\kappa}}{2}\right),\ & \text{for}\ \tilde{a}^\nu_i-\frac{h_{\epsilon^\kappa}}{2}<x_i<\tilde{a}^\nu_i,\\
            1,\ & \text{for}\ \tilde{a}_i^\nu\leq x_i\leq \tilde{b}_i\nu,\\
            1-2h_{\epsilon^{\kappa}}^{-1}(x_i-\tilde{b}_i^\nu),\ & \text{for},\ \tilde{b}_i^\nu<x_i<\tilde{b}_i^\nu+\frac{h_{\epsilon^\kappa}}{2},\\
            0,\ & \text{for}\ x_i\geq \tilde{b}_i^\nu+\frac{h_{\epsilon^\kappa}}{2}.
            \end{cases}
        \end{align*}
        Note that $\Phi_{\nu,i}$ is a concatenation of two neural networks, since $2h_{\epsilon^{\kappa}}^{-1}\geq 1$. By Lemma \ref{Lemma high numbers}, we can realize the function $x\mapsto 2h_{\epsilon^{\kappa}}^{-1}x$ using a neural network $\Phi_{\epsilon}$ with $L_{\epsilon}\leq 1+\lceil\kappa\rceil c^{\gamma_\nu}$ layers, sparsity $s_{\epsilon}\leq 4\lceil\kappa\rceil c^{\gamma_\nu}+5$ and weights in $\mathcal{W}_{c^{\gamma_{\nu}}}$.  Thus, $\Phi_{\nu,i}$ has $L_{\nu_i}\leq 4+\lceil\kappa\rceil c^{\gamma_\nu}$ layers, sparsity $s_{\nu,i}\leq 32\lceil\kappa\rceil c^{\gamma_\lambda}+32$ and weights in $\mathcal{W}_{\lceil \kappa\rceil c^{\gamma_\nu}}$. 
        We then define 
        \[
            R(\Phi_\nu)(x):=\sigma\bigg(\sum_{i=1}^dR(\Phi_{\nu,i})(x)+R(\Phi_{\gamma_{\nu}})(x)-d\bigg).
        \]
        For $x\in[0,1]^d$ we have $R(\Phi_\nu)(x)=1$ if $x\in \tilde{D}_\nu$ and $\iota_\nu x_{j_\nu}\leq \hat{\gamma}_\nu(x_{-j_\nu})$. Otherwise $0\leq R(\Phi_\nu(x))<1$ holds. Note that by regarding the construction of $\tilde{D}$, we have $R(\Phi_{\nu_1})R(\Phi_{\nu_2})=0$ for $\nu_1\neq\nu_2$. In order to construct the sum, we used a parallelization of the networks $\Phi_{\nu,1},\dots,\Phi_{\nu,d}$ and $\Phi_{\gamma_\nu}$. Thus, the network $\Phi_\nu$ has $L_\nu\leq C_1^\nu\lceil\log\epsilon^{-1}\rceil$ layers, sparsity $s_\nu \leq C_2^\nu\epsilon^{-\rho}\lceil\log\epsilon^{-1}\rceil$
        and weights in $\mathcal{W}_{c^\nu}$ with $c^\nu=C_3^\nu+C_4^\nu\lceil\log\epsilon^{-1}\rceil$ for some constants  $C_1^\nu,C_2^\nu>0$, $C_3^\nu,C_4^\nu\in\mathbb{N}$\footnote{Note that the notation suggests that the constants differ depending on $\nu=1,\dots,u$. However, due to the construction, this is not the case. The superscripts in the notations above are given in order to describe where the constant comes from. For our analysis below, it does not make a difference if the constants change with $\nu$ or not.}. Note that Lemma \ref{Lemma high numbers} was used to construct $d\geq 1$. The realization of the final network is given by
        \[
            R(\Phi)(x) = \sum_{\nu : \tilde{D}_\nu\neq\emptyset} R(\Phi_{\nu})(x).
        \]
        Define $G:=R(\Phi)^{-1}(1)$. We now verify the desired properties.\newline 
        
        \noindent We begin by finding constants $C_1',C_2',C_3'>0$ such that 
        \[
            G\in \mathcal{N}_{C_1'L_0(\tau_n^{-1}),C_2's_0(\tau_n^{-1}),C_3'c_0(\tau_n^{-1})}.
        \]
        Clearly, this realization $R(\Phi)$ can be achieved with
        \begin{align*}
            L & \leq \max_{\nu:\tilde{D}_\nu\neq\emptyset}(C_1^\nu)\lceil\log\epsilon^{-1}\rceil+1=\max_{\nu:\tilde{D}_\nu\neq\emptyset}(C_1^\nu+1)\lceil\log\tau_n\rceil =:C_1'L_0(\tau_n^{-1})
        \end{align*}
        layers, sparsity
        \begin{align*}
            s & \leq 2u(C_2^\nu\epsilon^{-\rho}\lceil\log\epsilon^{-1}\rceil +Ld)\leq 2r(C_2^\nu\epsilon^{-\rho}\lceil\log\epsilon^{-1}\rceil +Ld)= C_2's_0(\tau_n^{-1})
        \end{align*}
        and weights in $\mathcal{W}_{C_3'c_0(\tau_n)}$ with
        \begin{align*}
            C_3'c_0(\tau_n^{-1})\geq C_3^\nu+C_4^\nu\lceil\log\epsilon^{-1}\rceil.
        \end{align*}
        with suitably chosen $C_1',C_2'>0$, $C_3'\in\mathbb{N}$. Note that the constants do not depend on $u$. \newline
        
        \noindent Next, we show that the set $G:=R(\Phi)^{-1}(1)$ satisfies the desired approximation property $d_{f_\mathbb{Q}}(G,G^*_\mathbb{Q})\leq\tau_n^{-\kappa}$. First, for $\nu=1,\dots,u$ define $E_\nu$ as follows. Let
        \begin{align*}
            E_\nu := \bigcup_{j=1}^d\left(\prod_{i=1}^{j-1}[0,1]\right)\times\left(\left[a_j^\nu,\tilde{a_j}^\nu\right]\cup\left[\tilde{b_j}^\nu,b_j^\nu\right]\right)\times\left(\prod_{i=j+1}^{d}[0,1]\right).
        \end{align*}
        It is easy to see that $\tilde{D}_\nu=\emptyset$ implies $D_\nu\subseteq E_\nu$. Set $E:=\bigcup_{\nu=1}^uE_\nu\cup\tilde{D}_\nu$. Figure \ref{fig:Menge_C} shows $E$ in the example considered in figures \ref{fig:Menge} and \ref{fig:Menge_C}.
        \begin{figure}
            \begin{center}
                \includegraphics[scale=0.3]{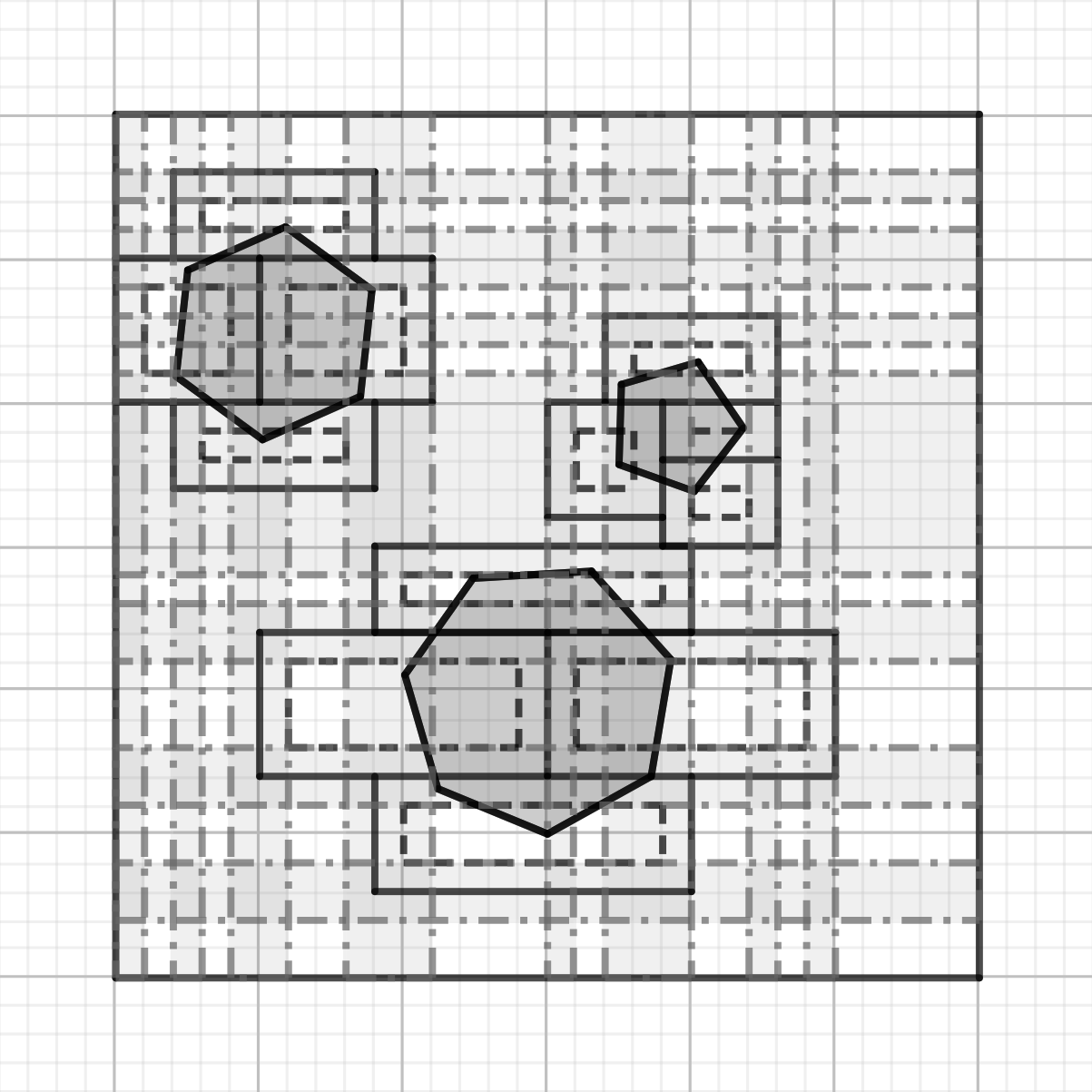}
            \end{center}
		    \caption{The set $E$ when considering the example from figures \ref{fig:Menge} and \ref{fig:Menge_D}. The light grey set represents $E$. Note that $E$ covers a majority of the space, since $\epsilon=\tau_n^{-1}$ is quite large in this example. Observe that $G_\mathbb{Q}^*,G\subseteq E$.} \label{fig:Menge_C}
	    \end{figure}
        Clearly $G_\mathbb{Q}^*,G\subseteq E$. Thus
        \begin{align*}
            d_{f_\mathbb{Q}}(G_\mathbb{Q},G) & = \int_{G_\mathbb{Q}^*\Delta G}|2f_\mathbb{Q}(x)-1|\mathbb{Q}_X(\mathrm{d}x)\\
            & \leq M\left(\sum_{\nu=1}^u\int_{E_\nu}1\mathrm{d}x+\sum_{\nu=1}^u\int_{(G_\mathbb{Q}^*\Delta G)\cap \tilde{D}_\nu}|2f_\mathbb{Q}(x)-1|\mathrm{d}x\right)\\
            & =:M\big( (I)+(II)\big).
        \end{align*}
        We need to bound both terms. For $(I)$ we observe that by construction $h_{\epsilon^\kappa}\leq 2\tau_n^{\kappa}$ we have
        \begin{align*}
            (I) & \leq \sum_{\nu=1}^{u}\sum_{j=1}^dh_{\epsilon^\kappa} \leq 2rd\tau_n^{\kappa}.
        \end{align*}
        The calculations for the second term are a bit more involved. First, observe that by construction of $G$, for all $\nu=1,\dots,u$ we have
        \begin{align*}
            & \int_{(G_\mathbb{Q}^*\Delta G)\cap \tilde{D}_\nu}|2f_\mathbb{Q}(x)-1|\mathrm{d}x\\
            & = \int_{\prod_{i\neq j_{\nu}}[\tilde{a}^\nu_i,\tilde{b}_i^\nu]}\int_{a_\nu(x_{-j_\nu})}^{b_\nu(x_{-j_\nu})}|2f_\mathbb{Q}(x)-1|\mathrm{d}x_{j_\nu}\mathrm{d}x_{-j_\nu}.
        \end{align*}
        where
        \begin{align*}
            b_\nu(x_{-j_\nu}) & :=\begin{cases}
                \tilde{a}^\nu_{j_\nu},\ & \text{if}\ \hat{\gamma}_\nu(x_{-j_\nu}),\gamma_\nu(x_{-j_\nu})<\tilde{a}^\nu_{j_\lambda} ,\\
                \tilde{b}_{j_\nu}^\nu,\ & \text{if}\ \tilde{b}^\nu_{j_\nu}<\hat{\gamma}_\nu(x_{-j_\nu}),\gamma_\nu(x_{-j_\nu}),\\
                \max\{\hat{\gamma}_\nu(x_{-j_\nu}),\gamma_\nu(x_{-j_\nu})\},\ & \text{otherwise} ,
            \end{cases}\\
            a_\nu(x_{-j_\nu}) & :=\begin{cases}
                \tilde{a}^\nu_{j_\nu},\ & \text{if}\ \hat{\gamma}_\nu(x_{-j_\nu}),\gamma_\nu(x_{-j_\nu})<\tilde{a}^\nu_{j_\lambda} ,\\
                \tilde{b}_{j_\nu}^\nu,\ & \text{if}\ \tilde{b}^\nu_{j_\nu}<\hat{\gamma}_\nu(x_{-j_\nu}),\gamma_\nu(x_{-j_\nu}),\\
                \min\{\hat{\gamma}_\nu(x_{-j_\nu}),\gamma_\nu(x_{-j_\nu})\},\ & \text{otherwise} .
            \end{cases}
        \end{align*}
        Let $x_{-j_\nu}\in\prod_{i\neq j_\nu}[\tilde{a}_i^\nu,\tilde{b}_i^\nu]$ be fixed. We have the following cases.
        \begin{itemize}
            \item Assume $\gamma_\nu(x_{-j_\nu})\geq \tilde{b}^\nu_{j_\nu}$. Then by construction we have 
            \[
                \hat{\gamma}_\nu(x_{-j_\nu})\geq \gamma_\nu(x_{-j_\nu})\geq \tilde{b}^\nu_{j_\nu}
            \] 
            and thus
            \begin{align*}
                \int_{a_\nu(x_{-j_\nu})}^{b_\nu(x_{-j_\nu})}|2f_\mathbb{Q}(x)-1|\mathrm{d}x_{j_\nu}=\int_{\tilde{b}^\nu_{j_\nu}}^{\tilde{b}^\nu_{j_\nu}}|2f_\mathbb{Q}(x)-1|\mathrm{d}x_{j_\nu}=0.
            \end{align*}
            \item Assume $\gamma_\nu(x_{-j_\nu})\leq \tilde{a}^\nu_{-j_\nu}$. Then by construction we have 
            \[
                \hat{\gamma}_\nu(x_{-j_\nu})\geq \gamma_\nu(x_{-j_\nu})\geq \tilde{a}^\nu_{j_\nu}
            \] 
            and thus
            \begin{align*}
                \int_{a_\nu(x_{-j_\nu})}^{b_\nu(x_{-j_\nu})}|2f_\mathbb{Q}(x)-1|\mathrm{d}x_{j_\nu}=\int_{\tilde{a}^\nu_{j_\nu}}^{\tilde{a}^\nu_{j_\nu}}|2f_\mathbb{Q}(x)-1|\mathrm{d}x_{j_\nu}=0.
            \end{align*}
            \item Assume $ \gamma_\nu(x_{-j_\nu})\in (\tilde{a}_{j_\nu}^\nu,\tilde{b}_{j_\nu}^\nu)$, by construction
            \[
            x^*_:=(x_1,\dots,x_{j_\nu-1},\gamma_\nu(x_{-j_\nu}),x_{j_\nu+1},\dots,x_d)\in\partial G_\mathbb{Q}^*.
            \]
            Consider $\gamma_\nu(x_{-j_\nu})\leq x_{j_\nu}\leq b_\nu(x_{-j_\nu})$. Let
            \[
                x=(x_1,\dots,x_{j_\nu-1},x_{j_\nu},x_{j_\nu+1},\dots,x_d).
            \]
            Now, for $\beta=0$ we have
            \begin{align*}
            \int_{a_\nu(x_{-j_\nu})}^{b_\nu(x_{-j_\nu})}|2f_\mathbb{Q}(x)-1|\mathrm{d}x_{j_\nu} & \leq\int_{\gamma_\nu(x_{-j_\nu})}^{\hat{\gamma}_\nu(x_{-j_\nu})}1\mathrm{d}x_{j_\nu}\\
            & = \hat{\gamma}_\nu(x_{-j_\nu})-\gamma_\nu(x_{-j_\nu})\\
            & \leq \tau_n^{-\kappa}.
            \end{align*}
            For $\beta>0$, by definition of $\mathcal{K}^\mathcal{F}_{\mathbb{Q},\epsilon_1,\epsilon_2,,r,d}$ we have 
            \[
                |2f_\mathbb{Q}(x)-1|\leq \left|g_{\nu,x^*}\big(x_{j_\nu}-\gamma_\nu(x_{-j_\nu})\big)\right| 
            \]
            Let $m:=\max\{k\in\mathbb{N}\ |\ k<\beta\}$ and $\omega:=\beta-m$. Using a Taylor expansion, there exists $y_{j_\nu}\in\big(0,x_{j_\nu}-\gamma_\nu(x_{-j_\nu})\big)$ such that
            \begin{align*}
                g_{\nu,x^*}\big(x_{j_\nu}-\gamma_\nu(x_{-j_\nu})\big) & = g_{\nu,x^*}\big(x_{j_\nu}-\gamma_\nu(x_{-j_\nu})\big)-g_{\nu,x^*}(0)\\
                & = \sum_{i=1}^{m-1}\frac{1}{i!}\partial_{j_\nu}^ig_{\nu,x^*}(0)\big(x_{j_\nu}-\gamma_\nu(x_{-j_\nu})\big)^i\\
                & \ \ \ \ \ \ +\frac{1}{m!}\partial_{j_\nu}^mg_{\nu,x^*}(y_{j_\nu})\big(x_{j_\nu}-\gamma_\nu(x_{-j_\nu})\big)^m\\
                & = \frac{1}{m!}\partial_{j_\nu}^mg_{\nu,x^*}(y_{j_\nu})\big(x_{j_\nu}-\gamma_\nu(x_{-j_\nu})\big)^m.
            \end{align*}
            Note that we used the definition of $\mathcal{H}_{\beta,B}$ in the last equality for all $i\leq m<\beta$. Thus
            \begin{align*}
                & |2f_\mathbb{Q}(x)-1|\\
                & \leq \left|g_{\nu,x^*}\big(x_{j_\nu}-\gamma_\nu(x_{-j_\nu})\big)\right|\\ 
                & \leq \frac{1}{m!}\frac{|\partial_{j_\nu}^mg_{\nu,x^*}(y_{j_\nu})-\partial_{j_\nu}^mg_{\nu,x^*}(0)|}{(y_{j_\nu}-0)^\omega}\big(x_{j_\nu}-\gamma_\nu(x_{-j_\nu})\big)^{\beta}\\
                & \leq \frac{B}{m!}\big(x_{j_\nu}-\gamma_\lambda(x_{-j_\nu})\big)^{\beta}.
            \end{align*}
            Similarly, for $a_\nu(x_{-j\nu})\leq x'_{j_\nu}\leq \gamma_\nu(x_{-j_\nu})$ we obtain
            \begin{align*}
                |2f_\mathbb{Q}(x)-1| \leq \frac{B}{m!}\big(x_{j_\nu}-\gamma_\lambda(x_{-j_\nu})\big)^{\beta}.
            \end{align*}
            This implies
            \begin{align*}
                & \int_{a_\nu(x_{-j_\nu})}^{b_\nu(x_{-j_\nu})}|2f_\mathbb{Q}(x)-1|\mathrm{d}x_{j_\nu}\mathrm{d}x_{-j_\nu}\\
                & \leq\int_{\gamma_\nu(x_{-j_\nu})}^{\hat{\gamma}_\nu(x_{-j_\nu})}\frac{B}{m!}\big(x_{j_\nu}-\gamma_\nu(x_{-j_\nu})\big)^{\beta}\mathrm{d}x_{j_\nu}\mathrm{d}x_{-j_\nu}\\
                & = \frac{B}{m!(\beta+1)}\big(\hat{\gamma}_\nu(x_{-j_\nu})-\gamma_\nu(x_{-j_\nu})\big)^{\beta+1}\\
                & \leq \frac{B}{m!(\beta+1)}\tau_n^{-\kappa}.
            \end{align*}
        \end{itemize}
        Therefore, we have
        \begin{align*}
                & \int_{a_\nu(x_{-j_\nu})}^{b_\nu(x_{-j_\nu})}|2f_\mathbb{Q}(x)-1|\mathrm{d}x_{j_\nu}\mathrm{d}x_{-j_\nu}\leq\max\left\{\frac{B}{m!(\beta+1)},1\right\}\tau_n^{-\kappa}.
            \end{align*}
        for all $\nu=1,\dots,u$ and $x_{-j_\nu}\in\prod_{i\neq j_\nu}[\tilde{a}_i^\nu,\tilde{b}_i^\nu]$ which yields
        \[
            (II)\leq \max\left\{\frac{B}{m!(\beta+1)},1\right\}\tau_n^{-\kappa}.
        \]
        Thus
        \[
            d_f(G,G_\mathbb{Q}^*)\leq \left(2rd+\max\left\{\frac{B}{m!(\beta+1)},1\right\}\right)\tau_n^{-\kappa}
        \]
        which concludes the proof.
    \end{proof}

    The remainder of the proof of our main result is now simple.

    \begin{proof}[Proof of Theorem \ref{Theorem Main}]
        We check the requirements in Proposition \ref{Theorem Mammen}.\newline
        
        \noindent Conditions (i) and (ii) are clear.\newline
        
        \noindent Condition (iii) follows from Theorem \ref{Theorem Approximation Result}.\newline
        
        \noindent Lastly, we need to prove (iv). Let $n\geq N_0$ where $N_0$ is defined in Theorem \ref{Theorem Approximation Result}. By Lemma \ref{Lemma Anzahl der Elemente} we have
        \begin{align*}
            |\mathcal{N}_n| & = \big|\mathcal{N}_{C_1'L_{0}(\tau_n),C_2's_{0}(\tau_n),C_3'c_0(\tau_n^{-1})}\big|\\
            & \leq \big((dC_2's_{n}(\tau_n)\\
            & \ \ \ \ \ +\min\{C_1'L_{n}(\tau_n),C_2's_0(\tau_n^{-1})\}(C_2's_{0}(\tau_n)+1)^2)2^{C_3'c_0(\tau_n^{-1})+2}\big)^{C_2's_{0}(\tau_n)}.
        \end{align*}
        Inserting all variables yields
        \begin{align*}
            |\mathcal{N}_n| 
            & \leq \big(k_1\tau_n^{k_2}\log^2(\tau_n)\big)^{k_3\tau_n^{\rho}\log(\tau_n)}
        \end{align*}
        for some constants $k_1,k_2,k_3>0$. Thus
        \begin{align*}
            \log\big(|\mathcal{N}_n|\big) & \leq k_4\tau_n^{\rho}\log^2(\tau_n)
        \end{align*}
        for some constant $k_4>0$. By setting 
        \[
            \tau_n:=\frac{n^{\frac{1}{2\kappa+\rho-1}}}{\log^{\frac{2}{\rho}}(n)}
        \]
        we obtain assumption (iv) 
        \[
            \log|\mathcal{N}_n| \leq c_3n^{\frac{\rho}{2\kappa+\rho-1}}. 
        \]
    \end{proof}
    
\subsection{Proofs for Regular Boundaries}
    
    Next, we prove Lemma \ref{Lemma Theorem 5}. We first provide the corresponding statement from  \cite{schmidt2020nonparametric}. Lemma \ref{Lemma Theorem 5} is a reformulated version.
    
    \begin{theo}(Theorem 5 in  \cite{schmidt2020nonparametric})\label{Theorem 5}\newline
        For any function $f\in\mathcal{F}_{\beta,B,d}$ and any integers $m\geq 1$ and $N\geq \max\big\{(\beta+1)^d,B+1\big\}$ there exists a neural network $\Phi$ with
        \[
            L = 8 + (m + 5)(1 + \lceil d\log_2 d\rceil)
        \]
        layers, sparcity 
        \[
            s\leq 94d^2(\beta + 1)^{2d}N(m + 6)(1 + \lceil \log_2 d\rceil)
        \]
        and weights $|w|\leq 1$ such that
        \[
            \|R(\Phi)-f\|_\infty\leq (2B+1)3^{d+1}N2^{-m}+B2^{\beta}N^{-\frac{\beta}{d}}.
        \]
    \end{theo}
    
    \begin{proof}
    Theorem 5 in  \cite{schmidt2020nonparametric}.
    \end{proof}
    
    \begin{proof}[Proof of Lemma \ref{Lemma Theorem 5}]
        Let $N$ be the smallest integer satisfying \begin{align*}
            N\geq k_1\epsilon^{-\frac{d}{\beta}}\geq \max\big\{\big(B2^{\beta+2}\epsilon^{-1}\big)^{\frac{d}{\beta}},(\beta+1)^d, B+1\big\},
        \end{align*}
        where $k_1:=\max\big\{\big(B2^{\beta+2}\big)^{\frac{d}{\beta}},(\beta+1)^d, B+1\big\}$. Let $k_2$ be the smallest integer satisfying
        \[
            k_2\geq \log(k_1+1)+\Big(\frac{d}{\beta}+1\Big)+\log((2B+1)3^{d+1})+1
        \]
        and define $m:=k_2\lceil \log(\epsilon^{-1}\rceil\in\mathbb{N}$. Since $\epsilon<3^{-1}$ and thus $\log(\epsilon^{-1})\geq 1$ we have
        \begin{align*}
            m & \geq \bigg(\log(k_1+1)+\Big(\frac{d}{\beta}+2\Big)+\log((2B+1)3^{d+1})+1\bigg)\log(\epsilon^{-1})\\ 
            & \geq \log(k_1+1)+\Big(\frac{d}{\beta}+1\Big)\log(\epsilon^{-1})+\log((2B+1)3^{d+1})+2\\
            & \geq N+\log((2B+1)3^{d+1})+\log(\epsilon^{-1})+2.
        \end{align*}
        Theorem \ref{Theorem 5} implies the following. For any $f\in\mathcal{F}_{\beta,B,d}$ there exists a network $\tilde\Phi$ with
        \[
            L=8+(k_2\lceil\log \epsilon^{-1}\rceil + 5)(1+ \lceil \log_2 d\rceil)
        \]
        layers, sparsity
        \[
            s\leq 94d^2(\beta+1)^{2d}k_1\epsilon^{-\frac{d}{\beta}}(k_2\lceil\log \epsilon^{-1}\rceil + 6)(1 + \lceil \log_2 d\rceil )
        \]
        and weights $|w_i|\leq 1$ such that
        \[
            \|R(\tilde\Phi)-f\|_\infty\leq \frac{\epsilon}{2}.
        \]
        Note that 
        \begin{align*}
            L & \leq (8+(2k_2+5)(1+\lceil\log_2 d\rceil)\log\epsilon^{-1}=:c_1\log\epsilon^{-1},\\
            s & \leq \big(94d^2(\beta+1)^{2d}k_1(2k_2+6)(1+\lceil\log_2d\rceil\big)\epsilon^{-\frac{d}{\beta}}\log\epsilon^{-1}=:c_2\epsilon^{-\frac{d}{\beta}}\log\epsilon^{-1}.
        \end{align*}
        Let $V$ be defined as in the proof of Lemma \ref{Lemma Anzahl der Elemente}. Following the proof of Lemma 12 of  \cite{schmidt2020nonparametric}, we see that for any $g\leq \frac{\epsilon}{4(L+1)V}$ there is a neural network $\Phi$ with $L$ layers and sparsity $s$ such that
        \[
            \|R(\Phi)-R(\tilde\Phi)\|_\infty\leq\frac{\epsilon}{2}.
        \]
        where the nonzero weights of $\Phi$ are discretized with grid size $g$. Now, define
        \begin{align*}
            & \frac{\epsilon}{4(L+1)V}\\
            & =\frac{\epsilon}{4(L+1)\big(ds+L(s+1)^2\big)}\\
            & \geq \frac{\epsilon}{4(c_1\lceil\log \epsilon^{-1}\rceil+1)\big(dc_2\epsilon^{-\frac{d}{\beta}}\lceil\log\epsilon^{-1}\rceil+c_1\lceil\log \epsilon^{-1}\rceil(c_2\epsilon^{-\frac{d}{\beta}}\lceil\log\epsilon^{-1}\rceil+1)^2\big)}\\
            &\geq \frac{1}{4(c_1+1)(dc_2+c_1(c_2+1)^2}\epsilon^{2+2\frac{d}{\beta}}\\
            &\geq 2^{-\big(c_3+c_4\lceil\log(\epsilon^{-1})\rceil\big)}=:g.
        \end{align*}
        with
        \begin{align*}
            c_3 & := \lceil\log\big(4(c_1+1)(dc_2+c_1(c_2+1)^2\big)\rceil,\\
            c_4 & := \Big\lceil 2+2\frac{d}{\beta}\Big\rceil.
        \end{align*}
        Therefore, all weights are elements of $\mathcal{W}_c$ with $c=c_3+c_4\lceil\log\epsilon^{-1}\rceil$ and
        \[
            \|R(\Phi)-f\|_\infty\leq \|R(\Phi)-R(\tilde\Phi)\|_\infty+\|R(\tilde\Phi)-f\|_\infty\leq\frac{\epsilon}{2}+\frac{\epsilon}{2}=\epsilon.
        \]
    \end{proof}
    
    Lastly, we prove Lemma \ref{Lemma curse}. The extension to this case is similar to the extension in  \cite{schmidt2020nonparametric}.
    
    \begin{proof}[Proof of Lemma \ref{Lemma curse}]
        Let 
        \[
            \gamma=\gamma_r\circ\dots\circ\gamma_1\in\mathcal{G}_{r,t,\beta,B,d}
        \] 
        with $\gamma_i=(\gamma_{ij}\circ\iota_{ij})_{j=1}^{d_{i+1}}$. We first construct a candidate network $\Phi$ for $\gamma$. Then, we show that it approximates gamma well and satisfies the required properties.\newline 
        
        \noindent In order to construct a network that approximates $\gamma$ well, we first approximate $\gamma_{ij}$ and $\tau_{ij}$ using neural networks. The final network is constructed using concatenation and parallelization.\newline 
        Let $i=1,\dots,r$,  $j=1,\dots,d_i+1$ and $\epsilon_{i}>0$. Using Lemma \ref{Lemma Theorem 5} there exist constants $\epsilon^{i}_0,c^{i}_1,c^{i}_2>0, c^{i}_3,c^{i}_4\in\mathbb{N}$ such that there exists a neural network $\Phi_{ij}$ with $L^{ij}\leq c^{i}_1 \lceil\log(\epsilon_{i}^{-1})\rceil$ layers, sparsity $s^{ij}\leq c^{i}_2 \epsilon_{i}^{-\frac{t_i}{\beta_i}}\log(\epsilon_{i}^{-1})$ and weights in $\mathcal{W}_{c^i}$ with $c^{i}:=c^{i}_3+c^{i}_4\lceil\log(\epsilon_{i}^{-1})\rceil$ such that
        \[
            \|R(\Phi_{ij})(x)-\gamma_{ij}\|_\infty\leq\epsilon_{i}
        \]
        if $\epsilon_{i}<\epsilon_0^{i}$. Let $\hat{\gamma}_{ij}:=R(\Phi_{ij})(x)$. Additionally, the function $\iota_{ij}$ is the realization of a network with 0 Layers and sparsity $t_i$.\\
        Since concatenating and parallelizing networks using Lemma \ref{lemma para+concat} leads to linear transformations on the upper bounds on the Layers, sparsity and the constant $c'$, there exist constants $c_1',c_2'>0$, $c_3',c_4'\in\mathbb{N}$ such that the function 
        \[
            \hat{\gamma}=\hat{\gamma_r}\circ\dots\circ\hat{\gamma}_1
        \]
        with $\hat{\gamma}_i=(\hat{\gamma}_{ij}\circ\iota_{ij})_{j=1}^{d-1}$ is the realization of a network with
        \begin{align*}
            L & \leq c'_1 \max_{i=1,\dots,r}\lceil\log(\epsilon_{i}^{-1})\rceil\ \text{layers},\\
            s & \leq c'_2 \max_{i=1,\dots,r}\epsilon_{i}^{-\frac{t_i}{\beta_i}}\log(\epsilon_i^{-1})\ \text{sparsity},\\
            c' & :=c'_3+c'_4\max_{i=1,\dots,r}\lceil\log(\epsilon_{i}^{-1})\rceil
        \end{align*}  
        and weights in $\mathcal{W}_{c'}$.\newline
        
        \noindent Now, let $\epsilon>0$ be small enough. We show that $\hat{\gamma}$ approximates $\gamma$ well for suitabily chosen $\epsilon_i$. 
        Following Lemma 9 in  \cite{schmidt2020nonparametric} we have
        \begin{align*}
            \|\gamma-\hat{\gamma}\|_\infty & \leq C \sum_{i=1}^r\|\max_{j=1,\dots,d_{i+1}}|\gamma_{ij}-\hat{\gamma}_{ij}|\|_\infty^{\prod_{k=i+1}^{r}\min\{\beta_k,1\}}\\
            & \leq C\sum_{i=1}^r\epsilon_i^{\prod_{k=i+1}^{r}\min\{\beta_k,1\}}\\
            & \leq Cr\max_{i=1,\dots,r}\epsilon_i^{\prod_{k=1}^{i+1}\min\{\beta_k,1\}}
        \end{align*}
        for some constant $C>0$. Set
        \[
            \epsilon_i:=\left(\frac{\epsilon}{Cr}\right)^{\frac{1}{\prod_{k=1}^{i+1}\min\{\beta_k,1\}}}.
        \]
        First, this implies
        \[
            \|\gamma-\hat{\gamma}\|_\infty\leq \epsilon.
        \]
        Additionally, the network $\Phi$ has
        \begin{align*}
            L & \leq c'_1 \max_{i=1,\dots,r}\lceil\log(\epsilon_{i}^{-1})\rceil\leq c_1\lceil\log(\epsilon^{-1})\rceil\ \text{layers},\\
            s & \leq c'_2 \max_{i=1,\dots,r}\epsilon_{i}^{-\frac{t_i}{\beta_i}}\log(\epsilon_i^{-1})= c_2 \max_{i=1,\dots,r}\epsilon^{-\frac{t_i}{\beta_i\prod_{k=1}^{i+1}\min\{\beta_k,1\}}}\log(\epsilon^{-1})\\
            & \ \ \ \ \ \ \ \ \ \ \ \ \ \ \ \ \ \ \ \ \ \ \ \ \ \ \ \ \ \ \ \ = c_2 \epsilon^{-\rho}\log(\epsilon^{-1}) \ \text{sparsity},\\
            c' &  =c'_3+c'_4\max_{i=1,\dots,r}\lceil\log(\epsilon_{i}^{-1})\rceil\leq c_3+c_4\lceil\log(\epsilon^{-1})\rceil:=c
        \end{align*}  
        and weights in $\mathcal{W}_{c}$ for some constants $c_1,c_2>0$, $c_3,c_4\in\mathbb{N}$.
    \end{proof}
    
\section{Lower Bound}\label{Appendix - Lower Bound}
    
    We first prove Theorem \ref{Theorem Lower Bound}. The outline of the proof is similar to the proof of Theorem 3 in  \cite{mammen1999smooth}. However, the setting of Theorem \ref{Theorem Main} differs substantially from theirs. This leads to a new situation and new technical challenges to overcome in the proof of Theorem \ref{Theorem Lower Bound} .  
    
    \begin{proof}[Proof of Theorem \ref{Theorem Lower Bound}]
    
        By Hoelders inequality and condition (c) it is enough to consider the case $p=1$ and the first inequality. Let $\mathfrak{Q}_1\subseteq\mathfrak{Q}$ be a finite set of potential probability measures of $(X_1,Y_1),\dots,(X_n,Y_n)$. Then
        \begin{align*}
            \sup_{\mathbb{Q}\in\mathfrak{Q}}\ \mathbb{E}[d_\Delta(G_{n},G^*_\mathbb{Q})] & \geq \frac{1}{\#\mathfrak{Q}_1}\sum_{\mathbb{Q}\in\mathfrak{Q}_1}\mathbb{E}[d_\Delta(G_{n},G^*_\mathbb{Q})]
        \end{align*}
        Hence, it suffices to show that for any estimator $G_{n}$ we have
        \begin{align}
            \frac{1}{\#\mathfrak{Q}_1}\sum_{\mathbb{Q}\in\mathfrak{Q}_1}\mathbb{E}[d_\Delta(G_{n},G^*_\mathbb{Q})]\geq cn^{-\frac{1}{2\kappa-1+\rho}}\ \ \ \mathrm{a.s.},\label{Inequality Lower Bound 1}
        \end{align}
        for some constant $c>0$. We now define the set $\mathfrak{Q}_1$. Then, we prove $\mathfrak{Q}_1\subseteq\mathfrak{Q}$. Lastly, we show that $\mathfrak{Q}_1$ satisfies \eqref{Inequality Lower Bound 1}.\newline 
        
        \noindent Let $\phi:\mathbb{R}\rightarrow[0,1]$ be an infinitely many times differentiable function with the following properties:
        \begin{itemize}
            \item $\phi(t)=0$ for $|t|\geq 1$,
            \item $\phi(0)=1$.
        \end{itemize}
        Let $K\geq 2$ be an integer. For $i\in\{1,\dots,K\}^{d-1}$ define
        \begin{align*}
            \phi_i:[0,1]^{d-1}\rightarrow[0,1],\ \phi_i(y)=k_1 K^{-\beta_2}\prod_{j=1}^{d-1}\phi\left(K\left(y_j-\frac{2i_j-1}{K}\right)\right)
        \end{align*}
        for some $0<k_1$ small enough. Define
        \[
            W:=\prod_{i\in\{1,\dots,K\}^{d-1}}\{0,1\}.
        \]
        For $w\in W$ let
        \begin{align*}
            \gamma_w:[0,1]^{d-1}\rightarrow[0,1],\ \gamma_w(y)=\sum_{i\in\{1,\dots,K\}^{d-1}}w_i\phi_i(y).
        \end{align*}
        Now, for $w\in W$ define $\mathbb{Q}_w$ as follows. The marginal distribution $\mathbb{Q}_{w,X}$ is the uniform distribution on $[0,1]^d$ and 
        \begin{align*}
            f_{\mathbb{Q}_w}(x) & := \frac{1}{2}\left(1+k_2\left(\gamma_w(x_{-1})-x_{1}\right)^{\beta_1}\right)\mathbbm{1}\left(x_{1}\leq \gamma_w(x_{-1})\right)\\
            & \ \ \ \ \ \ \ \ \ \ +\frac{1}{2}\left(1-k_2x_{1}^{\beta_1}\right)\mathbbm{1}\left(0<x_{1}\leq \gamma_{1-w}(x_{-1})\right)\\
            & \ \ \ \ \ \ \ \ \ \ +\frac{1}{2}\left(1-k_{3}\left(x_{1}-\gamma_1(x_{-1})\right)^{\beta_1}\right)\mathbbm{1}(\gamma_1(x_{-1})<x_{1})
        \end{align*}
        for some $k_2,k_3>0$. Finally, let
        \[
            \mathfrak{Q}_1:=\big\{\mathbb{Q}_w\ \big|\ w\in W\big\}.
        \]
        We now show that $\mathfrak{Q}_1\subseteq\mathfrak{Q}$ by properly selecting the constants $c_1,k_1,k_2,k_3$ such that $f_{\mathbb{Q}_w}$ is well defined for all $w\in W$ and showing that $\mathfrak{Q}_1$ satisfies the conditions (a),(b),(c).\newline
        
        \noindent First of all, we choose $k_1,k_3$ small enough and (given $k_2>0$) $K_0$ large enough such that for all $K\geq K_0$ we have 
        \[
            \frac{1}{4}\leq f_{\mathbb{Q}_w}(x)\leq 1
        \]
        for all $x\in [0,1]^d$ and $w\in W$. 
        \begin{itemize} 
            \item[(a)] Clearly, for all $w\in W$ the marginal distribution of $\mathbb{Q}_w$ with respect to $X$ has a Lebesgue density which bounded by $1\leq M$. 
            \item[(b)] We need to show
            \begin{align*}
                G_{\mathbb{Q}_w}^* & =\Big\{x\in[0,1]^d\Big|\ f_{\mathbb{Q}_w}(x)\geq\frac{1}{2}\Big\}\in \mathcal{K}_{\mathbb{Q},\beta,B,\epsilon_1,\epsilon_2,r,d}^{\mathcal{F}_{\beta_2,B_2,d-1}}
            \end{align*}
            for all $w\in W$.
            \begin{enumerate}
                \item Clearly, by selecting $\nu=u=1$, $j=j_\nu=1$, $\iota_2=1$, $D_\nu=[0,1]^d$ and $\gamma=\gamma_w$ we have 
                \[
                    G_{\mathbb{Q}_w}^*=H_1=D_\nu\cap\big\{x\in[0,1]^d\ |\ \iota_2x_{1}\leq \gamma(x_{-1})\big\}.
                \]
                For $k_1$ small enough we also have $\gamma\in \mathcal{F}_{\beta_2,B_2,{d-1}}$ for all $w\in W$.
                \item clear.
                \item If $\beta_1>0$, for $w\in W$ and $x\in\partial G_{\mathbb{Q}_w}^*$ we have $x_{1}=\gamma_w(x_{-1})$. Let 
                \[
                    g_{\nu,x}:[0,1]\rightarrow\mathbb{R},\ g_{\nu,x}(y)=\max\{k_2,k_3\}y^{\beta_1}.
                \]
                Note that for $k_2,k_3$ small enough we have $g_{\nu,x}\in\mathcal{H}_{\beta_1,B_1}$. Additionally, we have 
                \begin{align*}
                    |2f_{\mathbb{Q}_w}(y)-1| & \leq g_{\nu,x}(y-x_{1}),\  \text{for}\ y\geq x_{1},\\
                    |2f_{\mathbb{Q}_w}(y)-1| & \leq g_{\nu,x}(x_{1}-y),\  \text{for}\ y\leq x_{1}.
                \end{align*}
                \item clear.
            \end{enumerate}
            This implies the assertion.
            \item[(c)] Let $w\in W$. For $\beta_1=0$ we have
            \begin{align*}
                d^{\kappa}_\Delta(G,G_{\mathbb{Q}_w}^*) & = d_\Delta(G,G_{\mathbb{Q}_w}^*)=\frac{1}{\min\{k_2,k_3\}}d_{f_{\mathbb{Q}_w}}(G,G_{\mathbb{Q}_w}^*)
            \end{align*}
            For $\beta_1>0$, there is an $\eta_0>0$ such that for all $0<\eta\leq\eta_0$ we have
            \begin{align*}
                & \lambda\Big(\big\{x\in [0,1]^d\ \big|\ |2f_{\mathbb{Q}_w}(x)-1|\leq\eta\big\}\Big)\\ 
                & \leq \lambda\bigg(\Big\{x \in [0,1]^d\ \Big|\ x_{1}\leq \gamma_w(x_{-1}),\ k_2(\gamma_w(x_{-1})-x_{1})^{\beta_1}\leq \eta\Big\}\\
                & \ \ \ \cup \Big\{x \in [0,1]^d\ \Big|\ x_{1}\leq \gamma_{1-w}(x_{-1}),\ k_2x_{1}^{\beta_1}\leq \eta\Big\}\\
                & \ \ \ \cup \Big\{x \in [0,1]^d\ \Big|\ \gamma_{1}(x_{-1}) \leq x_{1},\ k_3\Big(x_{1}-\gamma_{1}(x_{-1})\Big)^{\beta_1}\leq \eta\Big\}\bigg)\\
                & \leq \lambda\Bigg(\Bigg\{x \in [0,1]^d\ \Bigg|\ \gamma_2(x_{-1})-\frac{1}{k_2^{\frac{1}{\beta_1}}}\eta^{\frac{1}{\beta_1}}\leq x_{1}\leq \gamma_w(x_{-1})\Bigg\}\\
                & \ \ \ \cup \Bigg\{x \in [0,1]^d\ \Bigg|\ x_{1}\leq \frac{1}{k_2^{{\frac{1}{\beta_1}}}}\eta^{\frac{1}{\beta_1}}\Bigg\}\\
                & \ \ \ \cup \Bigg\{x \in [0,1]^d\ \Bigg|\ \gamma_{1}(x_{-1}) \leq x_{1},\ \gamma_{1}(x_{-1})+\frac{1}{k_3^{\frac{1}{\beta_1}}}\eta^{\frac{1}{\beta_1}}\Bigg\}\Bigg)\\
                & \leq \left(\frac{2}{k_2^{\frac{1}{\beta_1}}}+\frac{1}{k_3^{\frac{1}{\beta_1}}}\right)\eta^{\frac{1}{\beta_1}}.
            \end{align*}
            Following Proposition 1 in of \cite{tsybakov2004optimal} there exists $\tilde{c}_1,\tilde{\eta}_0>0$ such that
            \begin{align*}
                d_\Delta^\kappa(G,G_{\mathbb{Q}_w}^*)\leq \tilde{c}_1d_{f_{\mathbb{Q}_w}}(G,G_{\mathbb{Q}_w}^*)
            \end{align*}
            for all $G$ such that $d_\Delta(G,G_{\mathbb{Q}_w}^*)\leq \tilde{\eta}_0$. If $\tilde{\eta}_0\geq 1$ this implies the assertion with $c_1:=\tilde{c}_1$. If not, the assertion is implied by setting $c_1:=\frac{\tilde{c}_1}{\tilde{\eta}_0^\kappa}$.
        \end{itemize}
        Next we prove Inequality \eqref{Inequality Lower Bound 1}. For $w\in W$ write $\mathbb{Q}_w^n$ for the probability measure of the distribution of $(X_1,Y_1),\dots,(X_n,Y_n)$ when the underlying distribution is $\mathbb{Q}_w$. Define the product measure $\psi = \zeta \times \lambda$, where $\zeta$ is the counting measure on $\{0,1\}$. Note that  $\mathbb{Q}_w$ has a density with respect to $\psi$ which is given by
        \begin{align*}
            \mathrm{d}\mathbb{Q}_w=\frac{\mathrm{d}\mathbb{Q}_w}{\mathrm{d}\psi}(y,x):=\mathbbm{1}(y=1)f_w(x)+\mathbbm{1}(y=0)\cdot(1-f_w(x)).
        \end{align*}
        Assume $w_1,w_2\in W$ differ by only 1 entry. We obtain
        \[
            \int\min\{\mathrm{d}\mathbb{Q}^n_{w_{1}},\mathrm{d}\mathbb{Q}^n_{w_{2}}\}\mathrm{d}\psi=\int\min\{\mathrm{d}\mathbb{Q}^n_{0},\mathrm{d}\mathbb{Q}^n_{1}\}\mathrm{d}\psi,
        \]
        where for $s=0,1$ we write $\mathbb{Q}^n_{s}=\mathbb{Q}^n_{w^s}$ with
        \begin{align*}
            w^s_{i}:=\begin{cases}
                 s,\ & \text{for}\ i_1=\dots = i_{d-1}=1,\\
                 0,\ & \text{otherwise}
            \end{cases}
        \end{align*}
        for $i\in\{1,\dots,K\}^{d-1}$. Then, using Assouad's Lemma we get
        \begin{align*}
            & \frac{1}{\#\mathfrak{Q}_1}\sum_{\mathbb{Q}\in\mathfrak{Q}_1}\mathbb{E}[d_\Delta(G_{n},G^*_\mathbb{Q})]\\
            & \geq \frac{1}{2}K^{d-1}\lambda\bigg(\Big\{x\in [0,1]^d\ \Big|\ x_{1}\leq \phi_1(x_{-1})\Big\}\ \bigg)\int\min\{\mathrm{d}\mathbb{Q}^n_{0},\mathrm{d}\mathbb{Q}^n_{1}\}\mathrm{d}\psi\\
            & = \frac{1}{2}k_1K^{d-1-\beta_2}\int_{\mathbb{R}^{d-1}}\prod_{j=1}^{d-1}\phi(Kx_{j+1})\mathrm{d}x_{-1}\int\min\{\mathrm{d}\mathbb{Q}^n_{0},\mathrm{d}\mathbb{Q}^n_{1}\}\mathrm{d}\psi.
        \end{align*}
        We first bound the term $\int\min\{\mathrm{d}\mathbb{Q}^n_{0},\mathrm{d}\mathbb{Q}^n_{1}\}\mathrm{d}\psi$. By using the fact that
        \[
            \int \mathrm{d}\mathbb{Q}_{0}\mathrm{d}\psi=1
        \]
        and Hoelders inequality in the fourth row we calculate
        \begin{align*}
            & \int\min\{\mathrm{d}\mathbb{Q}^n_{0},\mathrm{d}\mathbb{Q}^n_{1}\}\mathrm{d}\psi\\
            & = 1-\frac{1}{2}\int |\mathrm{d}\mathbb{Q}^n_{0}-\mathrm{d}\mathbb{Q}^n_{1}|\mathrm{d}\psi\\
            &= 1-\frac{1}{2}\int \left|\sqrt{\mathrm{d}\mathbb{Q}^n_{0}}-\sqrt{\mathrm{d}\mathbb{Q}^n_{1}}\right|\cdot \left|\sqrt{\mathrm{d}\mathbb{Q}^n_{0}}+\sqrt{\mathrm{d}\mathbb{Q}^n_{1}}\right|\mathrm{d}\psi\\
            & \geq 1-\frac{1}{2}\left(\int \left(\sqrt{\mathrm{d}\mathbb{Q}^n_{0}}-\sqrt{\mathrm{d}\mathbb{Q}^n_{1}}\right)^2\mathrm{d}\psi\right)^\frac{1}{2}\left( \int\left(\sqrt{\mathrm{d}\mathbb{Q}^n_{0}}+\sqrt{\mathrm{d}\mathbb{Q}^n_{1}}\right)^2\mathrm{d}\psi\right)^{\frac{1}{2}}.
        \end{align*}
        By repeatedly using the fact that $\int \mathrm{d}\mathbb{Q}_{0}\mathrm{d}\psi=1$, this implies
        \begin{align*}
            & \int\min\{\mathrm{d}\mathbb{Q}^n_{0},\mathrm{d}\mathbb{Q}^n_{1}\}\mathrm{d}\psi\\
            & = 1-\frac{1}{2}\left(2\left(1-\int \sqrt{\mathrm{d}\mathbb{Q}^n_{0}\mathrm{d}\mathbb{Q}^n_{1}}\mathrm{d}\psi\right)\right)^\frac{1}{2}\left( 2\left(1+\int\sqrt{\mathrm{d}\mathbb{Q}^n_{0}\mathrm{d}\mathbb{Q}^n_{1}}\mathrm{d}\psi\right)\right)^{\frac{1}{2}}\\
            & = 1-\left(1-\left(\int \sqrt{\mathrm{d}\mathbb{Q}^n_{0}\mathrm{d}\mathbb{Q}^n_{1}}\mathrm{d}\psi\right)^2\right)^\frac{1}{2}\\
            & \geq 1-\left(1-\left(\int \sqrt{\mathrm{d}\mathbb{Q}^n_{0}\mathrm{d}\mathbb{Q}^n_{1}}\mathrm{d}\psi\right)^2+\frac{1}{4}\left(\int \sqrt{\mathrm{d}\mathbb{Q}^n_{0}\mathrm{d}\mathbb{Q}^n_{1}}\mathrm{d}\psi\right)^4\right)^\frac{1}{2}\\
            & = 1-\left(1-\frac{1}{2}\left(\int \sqrt{\mathrm{d}\mathbb{Q}^n_{0}\mathrm{d}\mathbb{Q}^n_{1}}\mathrm{d}\psi\right)^2\right)\\
            & = \frac{1}{2}\left(\int \sqrt{\mathrm{d}\mathbb{Q}^n_{0}\mathrm{d}\mathbb{Q}^n_{1}}\mathrm{d}\psi\right)^2.
        \end{align*}
        By independence we have
        \begin{align*}
            \int \sqrt{\mathrm{d}\mathbb{Q}^n_{0}\mathrm{d}\mathbb{Q}^n_{1}}\mathrm{d}\psi=\left(\int \sqrt{\mathrm{d}\mathbb{Q}_{0}\mathrm{d}\mathbb{Q}_{1}}\mathrm{d}\psi\right)^n.
        \end{align*}
        Additionally, observe that 
        \begin{align*}
            & \int \sqrt{\mathrm{d}\mathbb{Q}_{0}\mathrm{d}\mathbb{Q}_{1}}\mathrm{d}\psi\\ 
            & = \frac{1}{2}\int \mathrm{d}\mathbb{Q}_{0}\mathrm{d}\psi+\frac{1}{2}\int \mathrm{d}\mathbb{Q}_{1}\mathrm{d}\psi-\frac{1}{2}\int \left(\sqrt{\mathrm{d}\mathbb{Q}_{0}}-\sqrt{\mathrm{d}\mathbb{Q}_{1}}\right)^2\mathrm{d}\psi\\
            & = 1-\frac{1}{2}\int \left(\sqrt{\mathrm{d}\mathbb{Q}_{0}}-\sqrt{\mathrm{d}\mathbb{Q}_{1}}\right)^2\mathrm{d}\psi
        \end{align*}
        and
        \begin{align*}
            & \int \left(\sqrt{\mathrm{d}\mathbb{Q}_{0}}-\sqrt{\mathrm{d}\mathbb{Q}_{1}}\right)^2\mathrm{d}\psi\\
            & \leq \int \left(\sqrt{f_{w^0}(x)}-\sqrt{f_{w^1}(x)}\right)^2\mathrm{d}x+ \int \left(\sqrt{1-f_{w^0}(x)}-\sqrt{1-f_{w^1}(x)}\right)^2\mathrm{d}x\\
            & \leq 2\int \big(f_{w^0}(x)-f_{w^1}(x)\big)^2\mathrm{d}x+ 2\int \big(1-f_{w^0}(x)-\big(1-f_{w^1}(x)\big)\big)^2\mathrm{d}x\\
            & = 4\int \big(f_{w^0}(x)-f_{w^1}(x)\big)^2\mathrm{d}x
        \end{align*}
        where we used that $f_w(x)\geq\frac{1}{4}$ for all $x\in[0,1]^d$ and $w\in W$. Next, we calculate
        \begin{align*}
            & \int \big(f_{w^0}(x)-f_{w^1}(x)\big)^2\mathrm{d}x\\ 
            & \geq \frac{1}{4}\int_{[0,1]^{d-1}}\int_{0}^{\phi_j(x_{-1})}\Big(k_2\left(\phi_j(x_{-1})-x_{1}\right)^{\beta_1} + k_2x_{1}^{\beta_1}\Big)^2\mathrm{d}x_{1}\mathrm{d}x_{-1}\\
            & \geq \frac{k_2^2}{4}\int_{[0,1]^{d-1}}\int_{0}^{\phi_j(x_{-1})}\left(\phi_j(x_{-1})-x_{1}\right)^{2\beta_1}\mathrm{d}x_{1}\mathrm{d}x_{-1}\\
            & \ \ \ \ \ \  +\frac{k_2^2}{4}\int_{[0,1]^{d-1}}\int_{0}^{\phi_j(x_{-1})}x_{1}^{2\beta_1}\mathrm{d}x_{1}\mathrm{d}x_{-1}\\
            & = \frac{k_2^2}{4}(I_1+I_2).
        \end{align*}
        We need to control the terms $I_1$ and $I_2$. For the first we obtain
        \begin{align*}
            I_1 & := \int_{[0,1]^{d-1}}\int_{0}^{\phi_j(x_{-1})}\left(\phi_j(x_{-1})-x_{1}\right)^{2\beta_1}\mathrm{d}x_{1}\mathrm{d}x_{-1}\\
            & = \int_{[0,1]^{d-1}}\int_{0}^{\phi_j(x_{-1})}x_{1}^{2\beta_1}\mathrm{d}x_{1}\mathrm{d}x_{-1}\\
            & = \frac{1}{1+2\beta_1}\int_{[0,1]^{d-1}}\phi_j(x_{-1})^{1+2\beta_1}\mathrm{d}x_{-1}\\
            & \leq \frac{k_1^{1+2\beta_1}}{1+2\beta_1}K^{-\beta_2(1+2\beta_1)}\int_{\mathbb{R}^{d-1}}\prod_{j=1}^{d-1}\phi(Kx_{j+1})^{1+2\beta_1}\mathrm{d}x_{-1}\\
            & \leq 2\frac{k_1^{1+2\beta_1}}{1+2\beta_1}K^{-\beta_2(1+2\beta_1)-(d-1)}\\
            & = 2\frac{k_1^{1+2\beta_1}}{1+2\beta_1}K^{-\beta_2(2\kappa-1+\rho)}
        \end{align*}
        and similarly
        \begin{align*}
            I_2 & := \int_{[0,1]^{d-1}}\int_{0}^{\phi_j(x_{-1})}x_{1}^{2\beta_1}\mathrm{d}x_{1}\mathrm{d}x_{-1}\\
            & \leq 2\frac{k_1^{1+2\beta_1}}{1+2\beta_1}K^{-\beta_2(2\kappa-1+\rho)}.
        \end{align*}
        This implies
        \begin{align*}
            \int\min\{\mathrm{d}\mathbb{Q}^n_{0},\mathrm{d}\mathbb{Q}^n_{1}\}\mathrm{d}\psi\geq \frac{1}{2}\left(1-c^*K^{-\beta_2(2\kappa-1+\rho)}\right)^{2n}
        \end{align*}
        for some constant $c^*>0$. By setting $K:=n^{\frac{1}{\beta_2}\frac{1}{2\kappa-1+\rho}}$ we obtain
        \begin{align*}
            \int\min\{\mathrm{d}\mathbb{Q}^n_{0},\mathrm{d}\mathbb{Q}^n_{1}\}\mathrm{d}\psi\geq\frac{1}{2}\left(1-c^*\frac{1}{n}\right)^{2n}> c'
        \end{align*}
        for some constant $c'>0$ for $n$ large enough. Thus
        \begin{align*}
            & \frac{1}{\#\mathfrak{Q}_1}\sum_{\mathbb{Q}\in\mathfrak{Q}_1}\mathbb{E}[d_\Delta(G_{n},G^*_\mathbb{Q})]\\
            & = \frac{1}{2}k_1K^{d-1-\beta_2}\int_{\mathbb{R}^{d-1}}\prod_{j=1}^{d-1}\phi(Kx_{j+1})\mathrm{d}x_{-1}\int\min\{\mathrm{d}\mathbb{Q}^n_{0},\mathrm{d}\mathbb{Q}^n_{1}\}\mathrm{d}\psi\\
            & \geq \frac{1}{2}k_1K^{-\beta_2}\int_{\mathbb{R}^{d-1}}\prod_{j=1}^{d-1}\phi(x_{j+1})\mathrm{d}x_{-1}\cdot c'\\
            & \geq c n^{-\frac{1}{2\kappa-1+\rho}}
        \end{align*}
        for some constant $c>0$. This concludes the proof.
    \end{proof}
    
    Lastly, the proof of Theorem \ref{Theorem Lower Bound curse} is provided. The ideas used in the proof are very similar to those used in the proof of Theorem \ref{Theorem Lower Bound} above. We therefore only focus on the differences.
    
    \begin{proof}[Proof of Theorem \ref{Theorem Lower Bound curse}]
        As in the proof of Theorem \ref{Theorem Lower Bound} the strategy is to show that for any estimator $G_{n}$ we have
        \begin{align*}
            \frac{1}{\#\mathfrak{Q}_1}\sum_{\mathbb{Q}\in\mathfrak{Q}_1}\mathbb{E}[d_\Delta(G_{n},G^*_\mathbb{Q})]\geq cn^{\frac{1}{2\kappa-1+\rho}}\ \ \ \mathrm{a.s.},
        \end{align*}
        for some constant $c>0$ and some finite set $\mathfrak{Q}_1\subseteq\mathfrak{Q}$.
        Let $K\geq 2$ be an integer and let 
        \[
            i_{\mathrm{opt}}:=\underset{i=1,\dots,r_2}{\arg\ \max}\ \frac{t_i}{\beta_{2,i}^*}.
        \] 
        As in the proof of Theorem \ref{Theorem Lower Bound} define $\phi:\mathbb{R}\rightarrow[0,1]$ to be an infinitely many times differentiable function with the following two properties:
        \begin{itemize}
            \item $\phi(t)=0$ for $|t|\geq 1$,
            \item $\phi(0)=1$.
        \end{itemize}
        Note that $\phi^\alpha$ also fulfills both  properties for any $\alpha>0$, though it may not be infinitely many times differentiable. For $i\in\{1,\dots,K\}^{t_{i_{\mathrm{opt}}}}$ define
        \begin{align*}
            \phi_i:[0,1]^{d_1-1}\rightarrow[0,1],\ \phi_i(y)=k_1 K^{-\beta_{2,i_\mathrm{opt}}^*}\prod_{j=1}^{t_{i_{\mathrm{opt}}}}\phi^\alpha\left(K\left(y_j-\frac{2i_j-1}{K}\right)\right)
        \end{align*}
        for $\alpha:=\prod_{k=i_{\mathrm{opt}}}^{r_2}\min\{\beta_k,1\}$ and some $0<k_1$ small enough. Define
        \[
            W:=\prod_{i\in\{1,\dots,K\}^{d-1}}\{0,1\}.
        \]
        For $w\in W$ let
        \begin{align*}
            \gamma_w:[0,1]^{d_1-1}\rightarrow[0,1],\ \gamma_w(y)=\sum_{i\in\{1,\dots,K\}^{t_{i_{\mathrm{opt}}}}}w_i\phi_i(y).
        \end{align*}
        Now, for $w\in W$ we define $\mathbb{Q}_w$ as before. The marginal distribution $\mathbb{Q}_X$ is the uniform distribution on $[0,1]^d$ and 
        \begin{align*}
            f_{\mathbb{Q}_w}(x) & := \frac{1}{2}\left(1+k_2\left(\gamma_w(x_{-1})-x_{1}\right)^{\beta_1}\right)\mathbbm{1}\left(x_{1}\leq \gamma_w(x_{-1})\right)\\
            & \ \ \ \ \ \ \ \ \ \ +\frac{1}{2}\left(1-k_2x_{1}^{\beta_1}\right)\mathbbm{1}\left(0<x_{1}\leq \gamma_{1-w}(x_{-1})\right)\\
            & \ \ \ \ \ \ \ \ \ \ +\frac{1}{2}\left(1-k_{3}\left(x_{1}-\gamma_1(x_{-1})\right)^{\beta_1}\right)\mathbbm{1}(\gamma_1(x_{-1})<x_{1})
        \end{align*}
        for some $k_2,k_3>0$. Note that for $k_1>0$ small enough $\gamma:=\gamma_w\in\mathcal{G}_{r_2,t,\beta_2,B_2,d'}$ by defining
        \begin{align*}
            \gamma_i(y) & =(y_1,\dots,y_{t_i},0,\dots,0),\ \text{for}\ i< i_{\mathrm{opt}},\\
            \gamma_i(y) & =(\psi(y),0,\dots,0),\ \text{for}\ i = i_{\mathrm{opt}},\\
            \gamma_i(y) & = \left(k_1^{\alpha_i}y_1^{\min\{\beta_i,1\}},0,\dots,0\right),\ \text{for}\ i>i_{\mathrm{opt}},
        \end{align*}
        where
        \begin{align*}
            \psi(y) & :=\sum_{i\in\{1,\dots,K\}^{t_{i_{\mathrm{opt}}}}}w_ik_1^{\alpha_{i_\mathrm{opt}}} K^{-\beta_{2,i_\mathrm{opt}}}\prod_{j=1}^{t_{i_{\mathrm{opt}}}}\phi\left(K\left(y_j-\frac{2i_j-1}{K}\right)\right),\\
            \alpha_i & := \frac{1}{(r_2-i_{\mathrm{opt}}+1)\prod_{k=i+1}^{r_2}\min\{\beta_k,1\}}
        \end{align*}
        for $i_{\mathrm{opt}}\leq i\leq r_2$. The rest of the proof is analogous to the proof of Theorem \ref{Theorem Lower Bound}.
    \end{proof}

\noindent\textbf{Acknowledgments}\newline

\noindent First and foremost, I would like to thank Enno Mammen for supporting me with some helpful comments and inspiring insights during the creation of this paper. Additionally, many thanks goes to Munir Hiabu for assisting with comments during the final stages of the working process. 


\bibliographystyle{plain}

\bibliography{References}

\end{document}